\documentclass{patmorin}
\usepackage[T1]{fontenc}
\usepackage[utf8]{inputenc}
\usepackage{amsmath}
\usepackage{amsfonts}
\usepackage{amsthm}
\usepackage{graphicx}
\usepackage{enumerate}
\usepackage{amsfonts}
\usepackage{amsthm,mathtools}
\usepackage{paralist}
\usepackage{stmaryrd}
\usepackage{thm-restate}
\usepackage[usenames,dvipsnames,svgnames,table]{xcolor}
\usepackage{todonotes}

\usepackage[longnamesfirst,numbers,sort&compress]{natbib}

\usepackage[noabbrev,capitalise]{cleveref}
\crefname{thm}{Theorem}{Theorems}
\crefname{lem}{Lemma}{Lemmas}
\crefname{cor}{Corollary}{Corollaries}
\crefname{prop}{Proposition}{Propositions}
\crefname{conj}{Conjecture}{Conjectures}
\crefname{open}{Open Problem}{Open Problems}
\crefname{claim}{Claim}{Claims}
\crefname{enumi}{Item}{Items}
\crefname{obs}{Observation}{Observations}
\crefname{clm}{Claim}{Claims}
\crefname{openproblem}{Open Problem}{Open Problems}
\crefformat{equation}{(#2#1#3)}
\Crefformat{equation}{Equation #2(#1)#3}

\theoremstyle{plain}
\newtheorem{thm}{Theorem}
\newtheorem{lem}[thm]{Lemma}
\newtheorem{cor}[thm]{Corollary}
\newtheorem{obs}[thm]{Observation}

\newtheorem{clm}{Claim}
\theoremstyle{definition}

\definecolor{brew2}{rgb}{1.0, 1.0, 0.701960784314}
\definecolor{brew8}{rgb}{0.988235294118, 0.803921568627, 0.898039215686}

 \makeatletter
 \def\NAT@spacechar{~}
 \makeatother

\newcommand{\defin}[1]{\textcolor{Maroon}{\emph{#1}}}

\newcommand{\notex}[2]{}

\DeclareMathOperator{\dist}{dist}

\DeclareMathOperator{\tw}{tw}

\DeclareMathOperator{\qn}{qn}

\DeclarePairedDelimiter{\ceil}{\lceil}{\rceil}
\DeclarePairedDelimiter{\floor}{\lfloor}{\rfloor}

\newcommand{\tlabel}[1]{\label{t:#1}}
\newcommand{\tref}[1]{(T\ref{t:#1})}

\newcommand{\PP}{\mathcal{P}}
\renewcommand{\SS}{\mathcal{S}}

\renewcommand{\ge}{\geqslant}
\renewcommand{\le}{\leqslant}
\renewcommand{\geq}{\geqslant}
\renewcommand{\leq}{\leqslant}

\newcommand{\R}{\mathbb{R}}
\newcommand{\N}{\mathbb{N}}

\title{\MakeUppercase{Graph Product Structure for Non-Minor-Closed Classes}}

\author{%
Vida Dujmovi\'c%
\thanks{School of Computer Science and Electrical Engineering, University of Ottawa, Ottawa, Canada (\texttt{vida.dujmovic@uottawa.ca}). Research supported by NSERC and the Ontario Ministry of Research and Innovation.},\,\,
Pat Morin%
\thanks{School of Computer Science, Carleton University, Ottawa, Canada (\texttt{morin@scs.carleton.ca}). Research  supported by NSERC and the Ontario Ministry of Research and Innovation.},\,\, and
David R. Wood%
\thanks{School of Mathematics, Monash University, Melbourne, Australia (\texttt{david.wood@monash.edu}). Research supported by the Australian Research Council.}
}

\begin{document}
\begin{titlepage}
\maketitle

\begin{abstract}
Dujmovi\'c~et~al.~[\emph{J.~ACM}~'20] proved that every planar graph is isomorphic to a subgraph of the strong product of a bounded treewidth graph and a path. Analogous results were obtained for graphs of bounded Euler genus or apex-minor-free graphs. These tools have been used to solve longstanding problems on queue layouts, non-repetitive colouring, $p$-centered colouring, and adjacency labelling. This paper proves analogous product structure theorems for various non-minor-closed classes. One noteable example is $k$-planar graphs (those with a drawing in the plane in which each edge is involved in at most $k$ crossings). We prove that every $k$-planar graph is isomorphic to a subgraph of the strong product of a graph of treewidth $O(k^5)$ and a path. This is the first result of this type for a non-minor-closed class of graphs. It implies, amongst other results, that $k$-planar graphs have non-repetitive chromatic number upper-bounded by a function of $k$. All these results generalise for drawings of graphs on arbitrary surfaces. In fact, we work in a more general setting based on so-called shortcut systems, which are of independent interest. This leads to analogous results for certain types of map graphs, string graphs, graph powers, and nearest neighbour graphs.
\end{abstract}
\end{titlepage}
\pagenumbering{roman}
\tableofcontents
\newpage

\pagenumbering{arabic}
\section{Introduction}
\label{Introduction}

The starting point for this work is the following `product structure theorem' for planar graphs\footnote{In this paper, all graphs are finite and undirected. Unless mentioned otherwise, all graphs are also simple. For any graph $G$ and any set $S$ (typically $S\subseteq V(G)$), let $G[S]$  denote the graph with vertex set $V(G)\cap S$ and edge set $\{uv\in E(G) : u,v\in S\}$.  We use $G-S$ as a shorthand for $G[V(G)\setminus S]$. Undefined terms are in \citep{Diestel5}.} by \citet{dujmovic.joret.ea:planar} (with an improvement by \citet{UWY22}). A graph $G$ is \defin{contained} in a graph $X$ if $G$ is isomorphic to a subgraph of $X$.

\begin{thm}[\citep{dujmovic.joret.ea:planar,UWY22}]
\label{PlanarProduct}
Every planar graph is contained in:
\begin{compactenum}[(a)]
  \item $H\boxtimes P$ for some graph $H$ of treewidth at most $6$ and for some path $P$,
  \item $H\boxtimes P \boxtimes K_3$ for some graph $H$ of treewidth at most $3$ and for some path $P$.
\end{compactenum}
\end{thm}

Here $\boxtimes$ is the strong product,\!\footnote{The \defin{strong product} of graphs $A$ and $B$, denoted by $A\boxtimes B$, is the graph with vertex set $V(A)\times V(B)$, where distinct vertices $(v,x),(w,y)\in V(A)\times V(B)$ are adjacent if
  $v=w$ and $xy\in E(B)$, or
  $x=y$ and $vw\in E(A)$, or
  $vw\in E(A)$ and $xy\in E(B)$.}
and treewidth\footnote{For a tree $T$, a \defin{$T$-decomposition} of a graph $G$ is an indexed family $\mathcal{T}=(B_x:x\in V(T))$ of subsets of $V(G)$ indexed by the nodes of $T$ such that
(i) for every $vw\in E(G)$, there exists some node $x\in V(T)$ with $v,w\in B_x$; and
(ii) for every $v\in V(G)$, the induced subgraph $T[v] := T[\{x: v\in B_x\}]$ is connected. The \defin{width} of $\mathcal{T}$ is $\max\{|B_x|:x\in V(T)\}-1$.  A \defin{tree-decomposition} is a $T$-decomposition for any tree $T$. The \defin{treewidth} $\tw(G)$ of a graph $G$ is the minimum width of a tree-decomposition of $G$.  Treewidth is the standard measure of how similar a graph is to a tree. Indeed, a connected graph has treewidth 1 if and only if it is a tree. Treewidth is of fundamental importance in structural and algorithmic graph theory; see \citep{Reed03,HW17,Bodlaender-TCS98} for surveys.} is an invariant that measures how `tree-like' a given graph is; see \cref{ProductExample} for an example. Loosely speaking, \cref{PlanarProduct} says that every planar graph is contained in the product of a tree-like graph and a path. This enables combinatorial results for graphs of bounded treewidth to be generalised for planar graphs (with different constants).

\begin{figure}[!h]
\centering
\includegraphics{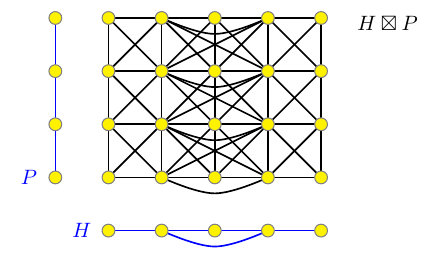}
\caption{Example of a strong product.
\label{ProductExample}}
\end{figure}

\noindent\cref{PlanarProduct} has been the key tool in solving the following well-known open problems:
\begin{compactitem}
\item \citet{dujmovic.joret.ea:planar} use it to prove that planar graphs have bounded queue-number (resolving a conjecture of \citet{HLR92}).
\item  \citet{dujmovic.esperet.ea:planar} use it to prove that planar graphs have bounded non-repetitive chromatic number (resolving a conjecture of \citet{AGHR-RSA02}).
\item \citet{DFMS21} use it to prove that planar graphs have $p$-centered chromatic number $O(p^2\log p)$ and give a matching lower bound.
\item \citet{DEJGMM21} use it to find asymptotically optimal adjacency labellings of planar graphs (resolving a problem of \citet{kannan.naor.ea:implicit}).
\item \citet{EJM} use it to show the existence of a `universal graph' with $n^{1+o(1)}$ vertices and edges that contains every $n$-vertex planar graph as an induced subgraph (resolving a problem of \citet{babai.chung.ea:on}).
\end{compactitem}
In addition, \cref{PlanarProduct} has been used to resolve or make substantial progress on a number of other problems on planar graphs, including $\ell$-vertex ranking~\citep{BDJM} and twin-width~\citep{BKW}.

All of these results hold for any graph class that has a product structure theorem analogous to \cref{PlanarProduct}; that is, for any graph class  $\mathcal{G}$ where every graph in $\mathcal{G}$ is contained in $H\boxtimes P\boxtimes K_\ell$ where $H$ has bounded treewidth, $P$ is a path, and $\ell$ is bounded.\footnote{It is easily seen that $\tw(H\boxtimes K_\ell) \leq (\tw(H)+1)\ell-1$, so we may assume that $\ell=1$ in this statement.} These applications motivate finding product structure theorems for other graph classes. \citet{dujmovic.joret.ea:planar} prove product structure theorems for graphs of bounded Euler genus\footnote{The \textit{Euler genus} of a surface with $h$ handles and $c$ crosscaps is $2h+c$. The \textit{Euler genus} of a graph $G$ is the minimum integer $g$ such that $G$ embeds in a surface of Euler genus $g$. Of course, a graph is planar if and only if it has Euler genus 0; see \citep{mohar.thomassen:graphs} for more about graph embeddings in surfaces.} and for apex-minor-free graphs,\footnote{A graph $M$ is a \textit{minor} of a graph $G$ if a graph isomorphic to $M$ can be obtained from a subgraph of $G$ by contracting edges. A class $\mathcal{G}$ of graphs is \defin{minor-closed} if for every graph $G\in\mathcal{G}$, every minor of $G$ is in $\mathcal{G}$. A minor-closed class is \defin{proper} if it is not the class of all graphs. For example, for fixed $g\geq 0$, the class of graphs with Euler genus at most $g$ is a proper minor-closed class. A graph $G$ is $t$-apex if it contains a set $A$ of at most $t$ vertices such that $G-A$ is planar. A 1-apex graph is \defin{apex}.  A minor-closed class $\mathcal{G}$ is apex-minor-free if some apex graph is not in $\mathcal{G}$.} and \citet{DEMWW22} do so for bounded-degree graphs in any minor-closed class. See \citep{DHJLW21} for a survey on this topic.

The main purpose of this paper is to prove product structure theorems for several non-minor-closed classes of interest. Our results are the first of this type for non-minor-closed classes.

\subsection{$k$-Planar Graphs}

We start with the example of $k$-planar graphs. A graph is \defin{$k$-planar} if it has a drawing in the plane in which each edge is involved in at most $k$ crossings, where no three edges cross at a single point (see \cref{k_planar_section} for a formal definition). Such graphs provide a natural generalisation of planar graphs, and are important in graph drawing research; see the recent bibliography on 1-planar graphs and the 140 references therein \citep{kobourov.liotta.ea:annotated}. It is well-known that the family of $k$-planar graphs is not minor-closed.  Indeed, 1-planar graphs may contain arbitrarily large complete graph minors~\citep{dujmovic.eppstein.ea:structure}. Hence the preceding results are not applicable for  $k$-planar graphs. We extend \cref{PlanarProduct} as follows.

\begin{thm}
\label{kPlanarProduct}
Every $k$-planar graph is contained in $H\boxtimes P\boxtimes K_{18k^2+48k+30}$, for some graph $H$ of treewidth $\binom{k+4}{3}-1$ and for some path $P$.
\end{thm}

For the important special case of $1$-planar graphs, we obtain the following result with $H$ planar and a best possible bound on the treewidth of $H$.

\begin{thm}\label{1_planar_product}
Every $1$-planar graph is contained in $H\boxtimes P\boxtimes K_{7}$, for some planar graph $H$ of treewidth $3$ and for some path $P$.
\end{thm}

For at least two applications, the planarity of $H$ in \cref{1_planar_product} is critical in obtaining asymptotically tight bounds \cite{BDJM,DFMS21}. \cref{1_planar_product} is proven in the more general setting of $d$-framed multigraphs, which also implies a product structure theorem for $d$-map graphs in which $H$ is a planar graph with treewidth 3.

\subsection{Shortcut Systems}

The preceding result for $k$-planar graphs (\cref{kPlanarProduct}) is in fact a special case of a more general result that relies on the following definition. A non-empty set $\SS$ of non-trivial paths in a graph $G$ is a \defin{$(k,d)$-shortcut system} (for $G$) if:

\begin{compactitem}
\item every path in $\SS$ has length at most $k$,\footnote{A path of length $k$ consists of $k$ edges and $k+1$ vertices.  A path is \defin{trivial} if it has length 0 and \defin{non-trivial} otherwise.} and
\item for every $v\in V(G)$, the number of paths in $\SS$ that include $v$ as an internal vertex is at most $d$.
\end{compactitem}
Each path $P\in\SS$ is called a \defin{shortcut}; if $P$ has endpoints $v$ and $w$ then it is a \defin{$vw$-shortcut}. Given a graph $G$ and a $(k,d)$-shortcut system $\SS$ for $G$, let $G^{\SS}$ denote the supergraph of $G$ obtained by adding the edge $vw$ for each $vw$-shortcut in $\SS$.

This definition is related to $k$-planarity by the following observation:

\begin{obs}
\label{AddDummy}
Every $k$-planar graph is contained in $G^\SS$ for some planar graph $G$ and some $(k+1,2)$-shortcut system $\SS$ for $G$.
\end{obs}

The proof of \cref{AddDummy} is trivial: Given a $k$-plane embedding of a graph $G'$, create a planar graph $G$ by adding a dummy vertex at each crossing point. For each edge $vw\in E(G')$ there is a path $P$ in $G$ between $v$ and $w$ of length at most $k+1$ in which every internal vertex is a dummy vertex. Let $\SS$ be the set of such paths $P$. For each vertex $v$ of $G$, at most two paths in $\SS$ use $v$ as an internal vertex (since no original vertex of $G'$ is an internal vertex of a path in $\SS$). Thus $\SS$ is a $(k+1,2)$-shortcut system for $G$, such that $G'\subseteq G^\SS$. This idea can be pushed further to obtain a rough characterisation of $k$-planar graphs, which is interesting in its own right, and is useful for showing that various classes of graphs are $k$-planar (see \cref{Characterisation}).

The following theorem is one of the main contributions of the paper. It says that if a graph class $\mathcal{G}$ has a product structure theorem like \cref{PlanarProduct}(b), then so too does the class of graphs obtained by applying shortcut systems to graphs in $\mathcal{G}$.

\begin{thm}
\label{ShortcutProduct}
Let $G$ be a graph contained in $H\boxtimes P \boxtimes K_\ell$, for some graph $H$ of treewidth at most $t$ and for some path $P$. Let $\SS$ be a $(k,d)$-shortcut system for $G$. Then $G^\SS$ is contained in $J\boxtimes P\boxtimes K_{d\ell(k^3+3k)}$ for some graph $J$ of treewidth at most $\binom{k+t}{t}-1$ and some path $P$.
\end{thm}

\cref{PlanarProduct}(b), \cref{ShortcutProduct}, and \cref{AddDummy} imply \cref{kPlanarProduct} with $K_{6(k^3+3k)}$ instead of $K_{18k^2+48k+30}$. Some further observations presented in \cref{k_planar_section} lead to the improved result.  \cref{ShortcutProduct} is applicable for many graph classes in addition to $k$-planar graphs. Some examples are explored in \cref{examples}.

\subsection{Overview and Outline}

\cref{summary_table} summarizes existing results on product structure theorems for minor-closed graph classes and new results for non-minor-closed graph classes.

\begin{table}[h]
  \centering{
    \begin{tabular}{lccl}
      \hline
      {Graph class} & $\tw(H)$ & $\ell$ & {Reference} \\ \hline
      planar & $3^*$ & $3$ & \cite{dujmovic.joret.ea:planar} \\
      planar & $4^*$ & $2$ & \cite{bose.morin.ea:optimal} \\
      planar & $6^*$ & $1$ & \cite{UWY22} \\
      genus $g$ & $3^*$ & $\max\{2g,3\}$ & \cite{DHHW22} \\
      genus $g$ & $2g+6$ & 1 & \cite{UWY22} \\
      apex-minor-free & $O(1)$ & $O(1)$ & \cite{dujmovic.joret.ea:planar} \\
      % B-minor-free & $O(1)$ & $O(1)$ & \cite{dujmovic.joret.ea:planar} \\
      $k$-planar & $\binom{k+4}{3}-1$ & $18k^2 + 48k + 30$ & \cref{kPlanarProduct} \\
      $(g,k)$-planar & $\binom{k+4}{3}-1$ & $\max\{2g,3\}(6k^2+16k+10)$ & \cref{gkPlanarProduct} \\
      $(g,\delta)$-string & $\binom{\delta+4}{3}-1$ & $\max\{2g,3\}(\delta^4 + 4\delta^3 + 9\delta^2 + 10\delta+4)$ & \cref{StringProduct} \\
      $k$-nearest-neighbour & $O(k^6)$ & $O(k^4)$ & \cref{nn_product_structure} \\
      $d$-framed & $3^*$ & $d+ 3\floor{d/2} - 3$ & \cref{d_framed_product_stucture} \\
      $1$-planar & $3^*$ & $7$ & \cref{1_planar_product} \\
      $d$-map & $3^*$ & $d+ 3\floor{d/2} - 3$ & \cref{dMapProduct} \\
      $(g,d)$-map & $9$ & $\max\{2g,3\}\,(7d^2 -21d)$ & \cref{gdMapProduct} \\ \hline
      \multicolumn{4}{l}{$^*$ these bounds also apply to the simple treewidth \cite{knauer.ueckerdt:simple} of $H$.}
    \end{tabular}
  }
  \caption{Product structure theorems (of the form $G\subseteq H\boxtimes P\boxtimes K_\ell$).}
  \label{summary_table}
\end{table}

The remainder of this paper is organized as follows:
\begin{compactitem}
  \item In \cref{Structure} we prove our main result for shortcut systems (\cref{ShortcutProduct}), and its optimisations for $k$-planar graphs and their genus-$g$ generalisation (\cref{kPlanarProduct,gkPlanarProduct}).
  \item \Cref{examples} uses \cref{ShortcutProduct,kPlanarProduct,gkPlanarProduct} to derive product structure theorems for bounded-degree string graphs, powers of bounded-degree graphs, map graphs of bounded clique number, and $k$-nearest neighbour graphs.
  \item \Cref{FramedSection} proves a product structure theorem for $d$-framed multigraphs in which the graph $H$ has treewidth $3$ and is planar.  Using this result, we derive improved product structure theorems for $1$-planar graphs and planar map graphs, again with a planar graph $H$.
  \item \Cref{Applications} quickly surveys applications of product structure theorems and the consequences of the current work.
\end{compactitem}

%%%%%%%%%%%%%%%%%%%%%%%%%%%
\section{\boldmath Shortcut Systems and $k$-Planar Graphs}
\label{Structure}

This section proves \cref{ShortcutProduct} and its specialization to $k$-planar graphs, \cref{kPlanarProduct}. While strong products enable concise statements of the theorems in \cref{Introduction}, to prove such results it is helpful to work with layerings and partitions, which we now introduce.

\subsection{Layered Partitions}

A \defin{layering} of a graph $G$ is an ordered partition $\mathcal{L}=\langle L_0,L_1,\ldots\rangle $ such that for every edge $vw\in E(G)$, if $v\in L_i$ and $w\in L_j$ then $|j-i|\leq 1$.  For any partition $\PP=\{S_1,\ldots,S_p\}$ of $V(G)$, a \defin{quotient graph} $H=G/\PP$ has a $p$-element vertex set $V(H)=\{x_1,\ldots,x_p\}$ and $x_ix_j\in E(H)$ if and only if there exists an edge $vw\in E(G)$ such that $v\in S_i$ and $w\in S_j$. To highlight the importance of the quotient graph $H$, we call $\PP$ an \defin{$H$-partition} and write this concisely as $\PP=\{S_x : x\in V(H)\}$ so that each element of $\PP$ is indexed by the corresponding vertex in $H$.

For any partition $\PP$ of $V(G)$ and any layering $\mathcal{L}$ of $G$ we define the \defin{layered width} of $\PP$ with respect to $\mathcal{L}$ as $\max\{|L\cap P|: L\in\mathcal{L},\, P\in\PP\}$.  For any partition $\PP$ of $V(G)$, we define the \defin{layered width} of $\PP$ as the minimum, over all layerings $\mathcal{L}$ of $G$, of the layered width of $\PP$ with respect to $\mathcal{L}$.

These definitions relate to strong products as follows.

\begin{lem}[\citep{dujmovic.joret.ea:planar}]
\label{PartitionProduct}
For every graph $H$, a graph $G$ has an $H$-partition of layered width at most $\ell$ if and only if $G$ is contained in $H \boxtimes P \boxtimes K_\ell$ for some path $P$.
\end{lem}

As an example of the use of layered partitions, to prove \cref{PlanarProduct}(a),
\citet{dujmovic.joret.ea:planar} build on an earlier result of \citet{PS21} to show that every planar graph has an $H$-partition of layered width $1$ for some planar graph $H$ of treewidth at most $8$ (improved to $6$ in \citep{UWY22}). The proof is constructive and gives a simple quadratic-time algorithm for finding the corresponding partition and layering.\footnote{\citet{bose.morin.ea:optimal} have recently given linear time algorithms for computing layered $H$-partitions of planar graphs.}

By \cref{PartitionProduct}, \cref{ShortcutProduct} is equivalent to the following result, whose proof is the subject of the next section.

\begin{thm}
  \label{ShortcutPartition}
  Let $G$ be a graph having an $H$-partition of layered width $\ell$ in which $H$ has treewidth at most $t$ and let $\SS$ be a $(k,d)$-shortcut system for $G$.  Then $G^\SS$ has a $J$-partition of layered width at most $d\ell(k^3+3k)$ for some graph $J$ of treewidth at most $\binom{k+t}{t}-1$.
\end{thm}

\subsection{Shortcut Systems}

We now prove \cref{ShortcutPartition}.
For convenience, it will be helpful to assume that $\SS$ contains a length-1 $vw$-shortcut for every edge $vw\in E(G)$.  Since $G^\SS$ is defined to be a supergraph of $G$, this assumption has no effect on $G^{\SS}$ but eliminates special cases in some of our proofs.  For a given $H$-partition $\mathcal{H}:=\{B_x:x\in V(H)\}$ of a graph $G$, we are frequently interested in the unique vertex $x\in V(H)$ such that $B_x$ contains a particular vertex $v\in V(G)$.  We express this concisely by writing that ``$P_x$ is the part in $\mathcal{H}$ that contains $v$.''

Let $T$ be a tree rooted at some node $x_0\in V(T)$.  A node $a\in V(T)$ is a \defin{$T$-ancestor} of $x\in V(T)$ (and $x$ is a \defin{$T$-descendant} of $a$) if $a$ is a vertex of the path, in $T$, from $x_0$ to $x$.  Note that each node $x\in V(T)$ is a $T$-ancestor and $T$-descendant of itself.  We say that a $T$-ancestor $a\in V(T)$ of $x\in V(T)$ is a \defin{strict} $T$-ancestor of $x$ if $a\neq x$.
The \defin{$T$-depth} of a node $x\in V(T)$ is the length of the path, in $T$, from $x_0$ to $x$.  The \defin{lowest common $T$-ancestor} of a non-empty set $S\subseteq V(T)$ is the maximum $T$-depth node $a\in V(T)$ that is a $T$-ancestor of every node in $S$.  In the following, we treat any subtree $T'$ of $T$ as a rooted tree whose root is the lowest common $T$-ancestor of $V(T')$.

We begin with a standard technique that allows us to work with a normalised tree-decomposition whose tree has the same vertex set as the graph it decomposes:

\begin{lem}\label{nice-decomposition}
  For every graph $H$ of treewidth $t$, there is a rooted tree $T$ with $V(T)=V(H)$ and a width-$t$ $T$-decomposition $(B_x:x\in V(T))$ of $H$ that has following additional properties:
  \begin{compactenum}[(T1)]
    \item\tlabel{subtree-root} for each node $x\in V(H)$, the subtree $T[x]:=T[\{y\in V(T):x\in B_y\}]$ is rooted at $x$; and consequently
    \item\tlabel{ancestor-edge}\tlabel{last} for each edge $xy\in E(H)$, one of $x$ or $y$ is a $T$-ancestor of the other.
  \end{compactenum}
\end{lem}

\begin{proof}
  That \tref{subtree-root} implies \tref{ancestor-edge} is a standard observation: If two subtrees intersect, then one contains the root of the other.  Thus, it suffices to construct a width-$t$ tree-decomposition that satisfies \tref{subtree-root}.

  Begin with any width-$t$ tree-decomposition $(B_x:x\in V(T_0))$ of $H$ that uses some tree $T_0$ and has no empty bags.  Select any node $x_1\in V(T_0)$, add a leaf $x_0$, with $B_{x_0}=\emptyset$, adjacent to $x_1$ and root $T_0$ at $x_0$. (The sole purpose of $x_0$ is to ensure that every node $x$ for which $B_x$ is non-empty has a parent.)  Let $f:V(H)\to V(T)$ be the function that maps each $x\in V(H)$ onto the root of the subtree $T_0[x]:=T_0[\{y\in V(T_0): x\in B_y\}]$.  If $f$ is not one-to-one, then select some distinct pair $x,y\in V(H)$ with $a:=f(x)=f(y)$.  Subdivide the edge between $a$ and its $T$-parent by introducing a new node $a'$ with $B_{a'}=B_{a}\setminus\{x\}$. Now $f(y)=a'$ and $f(x)=a$, so this modification reduces the number of distinct pairs $x,y\in V(H)$ with $f(x)=f(y)$.  Repeatedly performing this modification will eventually produce a tree-decomposition $(B_x:x\in V(T_0))$ of $H$ in which $f$ is one-to-one.

  Next, remove the node $x_0$ from $T_0$ (so that $x_1$ becomes the new root of $T_0$).  Consider any node $a\in V(T_0)$ such that there is no vertex $x\in V(H)$ with $f(x)=a$.  In this case, $B_{a}\subseteq B_{a'}$ where $a'$ is the parent of $a$ since any $x\in B_a\setminus B_{a'}$ would have $f(x)=a$.  In this case, contract the edge $aa'$ in $T_0$, eliminating the node $a$.  Repeating this operation will eventually produce a width-$t$ tree-decomposition of $(B_x:x\in V(T_0))$ where $f$ is a bijection between $V(H)$ and $V(T_0)$.  Renaming each node $a\in V(T_0)$ as $f^{-1}(a)$ gives a tree-decomposition $(B_x:x\in V(T))$ with $V(T)=V(H)$.  By the definition of $f$, the tree-decomposition $(B_x:x\in V(T))$ satisfies \tref{subtree-root}.
\end{proof}

\begin{proof}[Proof of \cref{ShortcutPartition}]
  Let $\mathcal{L}:=\langle L_1,\ldots,L_h\rangle$ be a layering of $G$; let $\mathcal{Y}:=(Y_x: x\in V(H))$ be an $H$-partition of $G$ of layered width at most $\ell$ with respect to $\mathcal{L}$; and let $\mathcal{T}:=(B_x:x\in V(T))$ be a tree-decomposition of $H$ satisfying the conditions of \cref{nice-decomposition}.

For a node $x\in V(T)$, we say that a shortcut $P\in\SS$ \defin{crosses} $x$ if $Y_x$ contains an internal vertex of $P$.  In other words, if $P=v_0,\ldots,v_r$ and $\{v_1,\ldots,v_{r-1}\}\cap Y_x\neq\emptyset$.  We say that a vertex $v\in V(G)$ \defin{participates} in $x\in V(T)$ if $v\in Y_x$ or if $\SS$ contains a shortcut $P$ with $v\in V(P)$ and $P$ crosses $x$.  For each $v\in V(G)$, let $a(v)$ denote the lowest common $T$-ancestor of all the nodes $x\in V(T)$ in which $v$ participates.

\begin{clm}\label{short_path_in_h}
  For each $v\in V(G)$, $H$ contains a path $x_0,\ldots,x_r$ of length $r\le k-1$ with $v\in Y_{x_0}$ and $x_r:=a(v)$ such that $v$ participates in each of $x_0,\ldots,x_r$.
\end{clm}

\begin{proof}
  Let $x_0\in V(T)$ be the unique node with $v\in Y_{x_0}$ and let $x_r:=a(v)$.  If $x_0=x_r$ then the length-$0$ path with the single vertex $x_0=x_r$ satisfies the requirements of the lemma.  Assume now that $x_0\neq x_r$.

  Consider the subgraph $H_v:=H[\{x\in V(T):\text{$v$ participates in $x$}\}]$.  Since $v\in Y_{x_0}$, $v$ participates in $x_0$, so $x_0\in V(H_v)$.  We claim that, for each $x\in V(H_v)\setminus\{x_0\}$, $H_v$ contains a path $y_0,\ldots,y_s$ of length $s\le k-1$ from $y_0:=x_0$ to $y_s:=x$.  Indeed, since $v$ participates in $x\neq x_0$, $\mathcal{S}$ contains a shortcut $P$ with $v\in V(P)$ and that crosses $x$.  Therefore $P$ contains a path $v_0,\ldots,v_s$ from $v_0:=v$ to some $v_s\in Y_x$.   Furthermore $v_s$ is an internal vertex of $P$ and $P$ has length at most $k$, so $s\le k-1$.  For each $i\in\{0,\ldots,s\}$, let $Y_{y_i}$ be the part of $\mathcal{Y}$ that contains $v_i$.  For each $i\in\{1,\ldots,s\}$, the existence of $v_{i-1}v_{i}\in E(G)$ implies that $y_{i-1}=y_i$ or that $y_{i-1}y_i\in E(H)$.  Therefore $y_0,\ldots,y_s$ is a lazy walk\footnote{A \defin{lazy walk} in a graph $G$ is a walk in the pseudograph obtained by adding a self-loop at each vertex of $G$.} in $H_v$ of length at most $s\le k-1$.  This lazy walk contains a path from $y_0$ to $y_s$ of length at most $s\le k-1$, as promised.

  Since $H_v$ contains a path from $x_0$ to $x$ for each $x\in V(H_v)\setminus\{x_0\}$,  $H_v$ is connected.  Since $H_v$ is connected, $x_r$---which is the lowest common $T$-ancestor of $V(H_v)$---is itself a node of $V(H_v)$.  Therefore $H_v$ contains a path $x_0,\ldots,x_r$ of length at most $k-1$.  By the definition of $H_v$, $v$ participates in each of $x_0,\ldots,x_r$.
\end{proof}

For each $x\in V(T)$, define $S_x := \{v\in V(G): a(v)= x\}$. Since $a:V(G)\to V(T)$ is a function with domain $V(G)$ and range $V(T)$, $\PP:=(S_x : x\in V(T))$ is an indexed family of disjoint sets that cover $V(G)$. Let $J:=G^\SS/\PP$ denote the resulting quotient graph. We consider $V(J)\subseteq V(T)$, where each $x\in V(J)$ is the vertex obtained by contracting $S_x$ in $G^{\SS}$. (Any node $x\in V(T)$ with $S_x=\emptyset$ does not contribute a vertex to $J$.)

From this point onward, the plan is to show that:
\begin{compactenum}[(i)]
  \item $\PP$ has small layered width with respect to the layering $\mathcal{L}$ of $G$, and
  \item $J$ has small treewidth.
\end{compactenum}
Once we have established (i) and (ii), the result follows easily since a layering of $G^\SS$ is easily obtained from $\mathcal{L}$ by `compressing' groups of $k$ consecutive layers.  We begin with Step~(i):

\begin{clm}
  \label{general-width}
  For each $i\in\{1,\ldots,h\}$ and each $x\in V(J)$, $|S_x\cap L_i|\le d\ell(k^2+3)$.
\end{clm}

\begin{proof}
  % Recall that $S_x$ is defined by vertices that participate in $x$, and these are vertices that are either in $Y_x$ or in shortcuts that cross $Y_x$.
  It follows from \cref{short_path_in_h} that, for each $v\in V(H)$, $v$ participates in $x:=a(v)$.  Therefore, for each $x\in V(J)$, the contents of $S_x$ are limited to vertices that are in $Y_x$ or in shortcuts that cross $x$.  We say that a vertex $w\in Y_x$ \defin{contributes} a vertex $v\in S_x$ if $v=w$ or if some path in $\SS$ that contains $v$ has $w$ as an internal vertex.  We upper bound the number of vertices in $S_x\cap L_i$ by upper-bounding the number of vertices contributed to $S_x\cap L_i$ by each $w\in Y_x$.

  Refer to \cref{contribute}.  If $w\in Y_x\cap L_i$ and no path in $\SS$ includes $w$ as an internal vertex then $w$ contributes at most one vertex, itself, to $S_x\cap L_i$.  Otherwise, consider some path $P\in\SS$ that contains $w$ as an internal vertex.  If $w\in L_{i}$, then $P$ contributes at most $k+1$ vertices to $S_x\cap L_i$.  If $w\in L_{i-1}\cup L_{i+1}$, then $P$ contributes at most $k$ vertices to $S_x\cap L_i$. If $w\in L_{i-j}\cup L_{i+j}$ for $j\ge 2$, then $P$ contributes at most $k-j$ vertices to $S_x\cap L_i$.

  \begin{figure}[htbp]
    \begin{center}
      \includegraphics{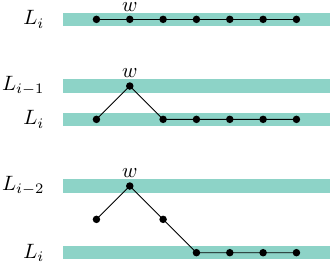}
    \end{center}
    \caption{A path $P$ containing an internal vertex $w\in Y_x\cap L_{i-j}$.}
    \label{contribute}
  \end{figure}

  For any $j$, the number of vertices $w\in L_{i+j}\cap Y_x$ is at most $\ell$. Each such vertex $w$ is an internal vertex of at most $d$ paths in $\SS$. Therefore,
  \[  |S_x\cap L_i|\le d\ell  \, \Big(k+1 + 2k + \sum_{j=2}^k 2(k-j)\Big)
  %= d(2k+1) + \sum_{i=1}^{k-2} i
      = d\ell(k^2 +3) \enspace . \qedhere
  \]
\end{proof}

We now proceed with Step~(ii), showing that $J$ has small treewidth. To accomplish this, we construct a small width tree-decomposition $\mathcal{C}:=(C_x:x\in V(T))$ of $J$ using the same tree $T$ used in the tree-decomposition $\mathcal{T}$ of $H$.  The following claim will be useful in showing that the resulting decomposition has small width.

\begin{clm}\label{i-ancestor}
  For each edge $xy\in E(J)$, one of $x$ or $y$ is a $T$-ancestor of the other.
\end{clm}

\begin{proof}
  If $xy\in E(J)$ then there exists a shortcut $P:=v_0,\ldots,v_r\in\SS$ with endpoints $v_0\in S_x$ and $v_r\in S_y$.\footnote{Recall that we have made the assumption that $\SS$ contains a length-$1$ $vw$-shortcut for each edge $vw\in E(G)$.}  For each $i\in\{0,\ldots,r\}$, let $Y_{x_i}$ be the part in $\mathcal{Y}$ that contains $v_i$.  Since $P$ is a connected subgraph of $G$, $H[\{x_0,\ldots,x_r\}]$ is a connected subgraph of $H$.  Since $\mathcal{T}$ is a normalized tree-decomposition of $H$, \tref{ancestor-edge} implies that some node $x_i$ is a $T$-ancestor of all nodes in $x_0,\ldots,x_r$.

  We claim that at least one of $v_0$ or $v_r$ participates in $x_i$.  If $i=0$ then $v_0\in Y_{x_0}=Y_{x_i}$ so $v_0$ participates in $x_i$. Similarly, if $i=r$ then $v_r$ participates in $x_r$.  Otherwise $i\in\{1,\ldots,r-1\}$ so $v_i$ is an internal vertex of the shortcut $P$, so $P$ crosses $x_i$.  Since $v_0,v_r\in V(P)$, this implies that $v_0$ and $v_r$ both participate in $x_i$.

  Suppose, without loss of generality, that $v_0$ participates in $x_i$.  Then $a(v_0)=x$ is a $T$-ancestor of $x_i$ which is a $T$-ancestor of $x_r$.  Finally $a(v_r)=y$ is a $T$-ancestor of $x_r$.  Therefore, both $x$ and $y$ are contained in the path from the root of $T$ to $x_r$, so at least one of $x$ or $y$ is a $T$-ancestor of the other.  This completes the proof of \cref{i-ancestor}.
\end{proof}

\begin{clm}
\label{general-bag-size}
The graph $J$ has a tree-decomposition in which every bag has size at most $\binom{k+t}{t}$.
\end{clm}

\begin{proof}
  For the tree-decomposition $(C_x:x\in V(T))$ of $J$ we use the same tree $T$ used in the tree-decomposition $(B_x:x\in V(T))$ of $H$. For each node $x$ of $T$, we define $C_x$ as follows: $C_x$ contains every $T$-ancestor $a$ of $x$ such that $J$ contains an edge $ax'$ where $x$ is a $T$-ancestor of $x'$ (including the possibility that $x=x'$).  Additionally, $C_x$ includes $x$ if $x\in V(J)$.
  \cref{i-ancestor} ensures that, for each edge $ax'\in E(J)$, both $a$ and $x'$ are contained in $C_{x'}$.  The connectivity of $T[a]:=T[\{x\in V(T):a\in C_x\}]$ follows from the fact that, for every node $x'\in V(T[a])$, every node $x$ on the path in $T$ from $a$ to  $x'$ is also a node of $T[a]$. Therefore $(C_x:x\in V(T))$ is indeed a tree-decomposition of $J$.

  It remains to show that each bag in this tree-decomposition has size at most $\binom{k+t}{t}$.  We will do this by appealing to an elegant result of \citet{PS21} on $k$-reachability in directed skeletons, which we now explain.

  Let $H^+$ denote the supergraph of $H$ that contains the edge $xy$ if and only if some bag $B_z\in\mathcal{T}$ contains both $x$ and $y$ (that is, $H^+$ is the chordal closure of $H$ with respect to $\mathcal{T}$).  Let $\overrightarrow{H}^+$ be the directed graph obtained by directing each edge $xy$ of $H^+$ in the direction $\overrightarrow{xy}$ so that $y$ is a $T$-ancestor of $x$. Observe that any directed path in $\overrightarrow{H}^+$ that begins at $x$ leads to a $T$-ancestor of $x$.  \citet[Lemma~13]{PS21} show that, for any $x\in V(\overrightarrow{H}^+)$, the number of strict $T$-ancestors of $x$ that can be reached by a directed path in $\overrightarrow{H}^+$ of length at most $k$ that begins at $x$ is at most $\binom{k+t}{t}-1$.

  Consider an arbitrary node $x\in V(T)$ and suppose that some strict $T$-ancestor $a$ of $x$ is contained in $C_x$.  We will show that $\overrightarrow{H}^+$ contains a directed path from $x$ to $a$ of length at most $k$.  By \cite[Lemma~13]{PS21}, this implies that the number of strict $T$-ancestors of $x$ contained in $C_x$ is at most $\binom{k+t}{t}-1$ and therefore $|C_x|\le\binom{k+t}{t}$.

  Since $a\in C_x$, there exists a $T$-descendant $x'$ of $x$ such that $ax'\in E(J)$.  Therefore, there exists a shortcut $P:=v_0,\ldots,v_r$ in $\SS$ with $v_0\in S_{x'}$ and $v_r\in S_a$.  For each $i\in\{0,\ldots,r\}$, let $Y_{x_i}$ be the part of $\mathcal{Y}$ that contains $v_i$. Since $x_r\in S_{a}$, $a(v_r)=a$.  By \cref{short_path_in_h}, $H$ contains a path $y_0,\ldots,y_s$ from $y_0:=x_r$ to $y_s:=a$ of length at most $k-1$ such that $v_r$ participates in each of $y_0,\ldots,y_s$.  We will use this to show that $H$ contains a path $z_0,\ldots,z_{s'}$ of length $s'\le s$, from $z_0:=x$ to $z_{s'}:=a$, and such that $v_r$ participates in each of $z_1,\ldots,z_{s'}$.   There are three cases to consider, depending on the location of $x_r$ relative to $x'$.

  % \begin{figure}
  %   \begin{center}
  %     \begin{tabular}{ccc}
  %       % (1) & (2) & (3) \\
  %       \includegraphics[page=2]{figs/width} &
  %       \includegraphics[page=3]{figs/width} &
  %       \includegraphics[page=4]{figs/width}
  %     \end{tabular}
  %   \end{center}
  %   \caption{Three cases in the proof of \cref{general-bag-size}}
  %   \label{width_figure}
  % \end{figure}

  \begin{enumerate}
    \item $x_r$ is a strict $T$-ancestor of $x$.  The shortcut $P=v_0,\ldots,v_r$ implies that $v_0$ participates in $x_{r-1}$.  Therefore $x'=a(v_0)$ is a $T$-ancestor of $x_{r-1}$, so $x$ is also a $T$-ancestor of $x_{r-1}$.  Since $x_{r-1}$ is a $T$-descendant of $x$ and $x_{r-1}x_r\in E(H)$, $x_r=y_0$ is contained in $B_x$. Therefore $xy_0$ is an edge of $H^+$, so the path $x,y_0,\ldots,y_s$ is a path in $H^+$ from $x$ to $a$ of length at most $s+1\le k$ and $v_r$ participates in each of $y_0,\ldots,y_s$.

    \item $x_r=x$.  In this case, $y_0,\ldots,y_s$ is a path from $x$ to $a$ of length at most $k-1\le k$ and $v_r$ participates in each of $y_0,\ldots,y_s$.

    \item $x_r$ is a strict $T$-descendant of $x$.
      Since  $y_0=x_r$ is a strict $T$-descendant of $x$ and $y_s=a$ is a $T$-ancestor of $x$, the path in $T$ from $y_0$ to $y_s$ includes $x$.  Since $H[\{y_0,\ldots,y_s\}]$ is connected, at least one of $y_0,\ldots,y_s$, say $y_i$, is contained in $B_x$. Therefore $xy_i$ is an edge of $H^{+}$.  Therefore $x,y_i,\ldots,y_s$ is a path in $H^+$ from $x$ to $a$ of length at most $s+1\le k$ and $v_r$ participates in each of $y_i,\ldots,y_s$.
  \end{enumerate}
  In each case, we obtain a path $z_0,\ldots,z_{s'}$ in $H$ of length $s'\le k$ from $z_0:=x$ to $z_{s'}:=a$ such that $v_r$ participates in each of $z_1,\ldots,z_{s'}$.  By definition, $a$ is a $T$-ancestor of $x$.  Since $v_r$ participates in each of $z_1,\ldots,z_{s'}$, $a=a(v_r)$ is also a $T$-ancestor of each of $z_1,\ldots,z_{s'}$.  Let $W$ be the sequence obtained from $z_0,\ldots,z_{s'}$ by replacing $z_i$ with the lowest common $T$-ancestor of $z_0,\ldots,z_i$, for each $i\in\{0,\ldots,s'\}$.  Then $W$ is a lazy walk in $H^+$ of length at most $k$ that begins at $z_0=x$ and ends at $z_{s'}=a$. By removing duplicates from $W$ we obtain a path in $H^+$ that is also a directed path in $\overrightarrow{H}^+$.  This completes the proof of \cref{general-bag-size}.
\end{proof}

At this point, the proof of \cref{ShortcutPartition} is almost immediate from \cref{general-width,general-bag-size}, except that the layering $\mathcal{L}$ of $G$ may not be a valid layering of $G^{\SS}$.  In particular, $G^{\SS}$ may contain edges $vw$ with $v\in L_i$ and $w\in L_{i+j}$ for any $j\in\{0,\ldots,k\}$.  To resolve this, we use a new layering $\mathcal{L}':=\langle L_0',\ldots,L_h'\rangle$ in which $L_i'=\bigcup_{j=ki}^{ki+k-1} L_i$.  This increases the layered width given by \cref{general-width} from $d\ell(k^2+3)$ to $d\ell(k^3+3k)$.  Therefore $G$ has an $H$-partition of layered width at most $d\ell(k^3+3k)$ in which $H$ has treewidth at most $\binom{k+t}{t}-1$, completing the proof of \cref{ShortcutPartition}.
\end{proof}

\subsection{\boldmath $k$-Planar Graphs}
\label{k_planar_section}

We first formally define $k$-planar graphs.  For a surface $\Sigma$,  a \defin{$\Sigma$-embedded graph} $G$ is a graph with $V(G)\subset\Sigma$ in which each edge $vw\in E(G)$ is a curve\footnote{A \defin{curve} in a surface $\Sigma$ is a continuous function $f:[0,1]\to \Sigma$. The points $f(0)$ and $f(1)$ are called the \defin{endpoints} of the curve.  When there is no danger of misunderstanding we treat a curve $f$ as the point set $\{f(t):0\le t\le 1\}$.  The \defin{interior} of $f$ is the point set $\{f(t):0<t<1\}$.} in $\Sigma$ with endpoints $v$ and $w$ and not containing any vertex of $G$ in its interior. The \defin{faces} of $G$ are the connected components of $\Sigma\setminus \bigcup_{vw\in E(G)} vw$.
If some point $p\in\Sigma$ is a common point in the interior of two distinct edges $vw$ and $xy$, then we say that $vw$ and $xy$ \defin{cross at} $p$ and we call the pair $(\{vw,xy\},p)$ a \defin{crossing} in $G$.  Without any loss of generality we may assume that no point in $\Sigma$ is contained in the interior of three or more edges of $G$.  This assumption can be enforced by local changes that do not change the number of crossings each edge is involved in.

% A \defin{crossing} in a $\Sigma$-embedded (multi)graph $G$ is a pair $(\{vw,xy\},p)$ where $vw$ and $xy$ are distinct edges of $G$ and $p$ is a point that is in the interior of both $vw$ and $xy$.

If an embedded graph $G$  has no crossings then it is \defin{non-crossing}.  When $G$ is non-crossing and $F$ is a face of $G$, we use \defin{$V(F)$} to denote the set of vertices in $G$ on the boundary of $F$.  When $F$ has a single boundary component, the \defin{facial walk} of $F$ is the walk consisting of the edges and vertices of $G$ on the boundary of $F$ in the cyclic order they are encountered while traversing the boundary of $F$.  If  $\Sigma=\R^2$ is the Euclidean plane and each edge of $G$ is involved in at most $k$ crossings, then $G$ is a $k$-plane graph.  A graph $G$ is \defin{$k$-planar} if it is isomorphic to some $k$-plane graph.  A $0$-plane graph is a \defin{plane} graph and a $0$-planar graph is a \defin{planar} graph.

As mentioned in \cref{Introduction}, \cref{PlanarProduct,ShortcutProduct} imply a product structure theorem for $k$-planar graphs. We get improved bounds as follows.

\begin{proof}[Proof of \cref{kPlanarProduct}]
  Let $G$ be a $k$-plane graph.  As in the proof of \cref{AddDummy}, let $G_0$ be the plane graph obtained by adding a dummy vertex at each crossing in $G$. In this way, each edge $vw\in E(G)$ corresponds naturally to a path $P_{vw}$ of length at most $k+1$ in $G_0$.  Let $\SS := \{P_{vw}: vw\in E(G)\}$. Observe that $\SS$ is a $(k+1,2)$-shortcut system for $G_0$ and that $G_0^{\SS}\supseteq G$.  Specifically, $G_0^{\SS}$ contains every edge and vertex of $G$ as well as the dummy vertices in $V(G_0)\setminus V(G)$ and their incident edges.

  Since $G_0$ is planar,  \cref{PlanarProduct}(b) and \cref{PartitionProduct} implies that $G_0$ has an $H$-partition of layered width 3 for some planar graph $H$ of treewidth at most 3.  Applying \cref{ShortcutPartition} to $G_0$ and $\SS$ immediately implies that $G$ (an arbitrary $k$-planar graph) has an $H$-partition of layered width $6((k+1)^3+3(k+1))$ for some graph $H$ of treewidth at most $\binom{k+4}{3}-1$.

  We can reduce the layered width of the $H$-partition of $G$ from $O(k^3)$ to $O(k^2)$ by observing that the dummy vertices in $V(G_0)\setminus V(G)$ do not contribute to the layered width of this partition.  In this setting, the proof of \cref{general-width} is simpler since each vertex $w\in Y_x$ contributes at most two vertices to $L_i\cap Y_x$.  More precisely, each path $P\in\SS$ containing an internal (dummy) vertex $w\in Y_x\cap (L_{i-j}\cup L_{i+j})$ contributes: (i)~at most two vertices to $S_x\cap L_i$ for $j\in\{0,\ldots,\floor{(k+1)/2}\}$; (ii)~at most one vertex to $S_x\cap L_j$ for $j\in\{\floor{(k+1)/2}+1,\ldots,k+1\}$; or (iii)~no vertices to $S_x\cap L_j$ for $j > k+1$.
  Redoing the calculation at the end of the proof of \cref{general-width} then yields
  \begin{align*}
  |S_x\cap Y_i| \le d\ell\left(
  2
  + 4\left\lfloor\tfrac{k+1}{2}\right\rfloor
  + 2\left\lceil\tfrac{k+1}{2}\right\rceil
  \right)
   =
  d\ell\left(
  2 + 2(k+1) + 2\left\lfloor\tfrac{k+1}{2}\right\rfloor
  \right)
   \le
  d\ell(3k+5)
  = 18k+30 \enspace .
  \end{align*}
  With this change, the layered width of the partition given by \cref{ShortcutPartition} becomes $(18k+30)(k+1)=18k^2+48k+30$.
  The result follows from \cref{PartitionProduct}.
\end{proof}

\cref{kPlanarProduct} shows that every $k$-planar graph is contained in $H\boxtimes P \boxtimes K_\ell$ for some graph $H$ with treewidth $O(k^3)$ where $\ell\leq O(k^2)$.  In some applications, the treewidth of $H$ is much more significant than the value of $\ell$, which leads to the following question:

\noindent\textbf{Open Problem:}
  Does there exist a function $\ell:\N\to\N$ and a universal constant $C$ such that every $k$-planar graph is contained in $H\boxtimes P \boxtimes K_{\ell(k)}$ for some graph $H$ with treewidth at most $C$?  Perhaps $C=3$, which is the case for planar graphs (\cref{PlanarProduct}(b)) and for $1$-planar graphs (\cref{1_planar_product}). Note that $C\geq 3$ even for planar graphs \citep{dujmovic.joret.ea:planar}.

\subsection{\boldmath $(g,k)$-Planar Graphs}
\label{gk_planar_section}

As mentioned in the introduction, product structure theorems have been established for several minor-closed classes in addition to planar graphs, including the following version of \cref{PlanarProduct} for graphs of bounded Euler genus.

\begin{thm}[\citep{dujmovic.joret.ea:planar,UWY22,DHHW22}]\label{GenusProduct}
  Every graph of Euler genus $g$ is contained in:
  \begin{compactenum}[(a)]
  \item $H  \boxtimes P$ for some graph $H$ of treewidth at most $2g+6$  and some path $P$.
  \item $H \boxtimes P \boxtimes K_{\max\{2g,3\}}$ for some planar graph $H$ of treewidth at most $3$ and for some path $P$.
  \end{compactenum}
\end{thm}

The definition of $k$-planar graphs naturally generalises to genus-$g$ surfaces. A $\Sigma$-embedded graph $G$ is \defin{$(\Sigma,k)$-plane} if every edge of $G$ is involved in at most $k$ crossings.  A graph $G$ is \defin{$(g,k)$-planar} if it is isomorphic to some $(\Sigma,k)$-plane graph, for some surface $\Sigma$ with Euler genus at most $g$. \cref{AddDummy} immediately generalises as follows:

\begin{obs}
\label{gAddDummy}
Every $(g,k)$-planar graph $G$ is contained in $G_0^\SS$ for some graph $G_0$ of Euler genus at most $g$ and some $(k+1,2)$-shortcut system $\SS$ for $G_0$. Moreover, $V(G) \subseteq V(G_0)$ and for every edge $vw \in E(G)$ there is a $vw$-path $P$ in $G_0$ of length at most $k+1$, such that every internal vertex in $P$ has degree at most $4$ in $G_0$.
\end{obs}

Theorems~\ref{ShortcutProduct}  and \ref{GenusProduct}(b) imply a product structure theorem for $(g,k)$-planar graphs. The resulting bounds are improved by the following theorem, which is proved using exactly the same approach used in the proof of \cref{kPlanarProduct} (applying \cref{GenusProduct}(b) instead of \cref{PlanarProduct}(b)). We omit repeating the details.

\begin{thm}
\label{gkPlanarProduct}
Every $(g,k)$-planar graph is contained in $H\boxtimes P \boxtimes K_\ell$ for some graph $H$ with $\tw(H) \leq \binom{k+4}{3}-1$, where $\ell:=\max\{2g,3\}\cdot(6k^2+16k+10)$.
\end{thm}

\subsection{Rough Characterisation}
\label{Characterisation}

\cref{gAddDummy} shows that $(g,k)$-planar graphs can be obtained by a shortcut system applied to a graph of Euler genus $g$, where internal vertices on the paths have bounded degree. This observation and the following converse result together provide a rough characterisation of $(g,k)$-planar graphs, which is interesting in its own right, and is useful for showing that various classes of graphs are $(g,k)$-planar.

\begin{lem}
  \label{DrawG}
  Fix integers $g\geq 0$ and $k,\Delta\geq 2$.
  Let $G_0$ be a graph of Euler genus at most $g$. Let $G$ be
  a graph with $V(G) \subseteq V(G_0)$ such that for every edge $vw \in
  E(G)$ there is a $vw$-path $P_{vw}$ in $G_0$ of length at most $k$, such
  that every internal vertex on $P_{vw}$ has degree at most $\Delta$ in
  $G_0$. Then $G$ is $(g, 2k(k+1)\Delta^{k} )$-planar.
\end{lem}

\begin{proof}
  For a vertex $x$ of $G_0$ with degree at most $\Delta$, and for $i\in\{1,\dots,k-1\}$, say a vertex $v$ is \defin{$i$-close} to $x$ if there is a $vx$-path $P$ in $G_0$ of length at most $i$ such that every internal vertex in $P$  has degree at most $\Delta$ in $G_0$.
  For each edge $vw$ of $G$, say that $vw$ \defin{passes through} each internal vertex on $P_{vw}$.  Say $vw$ passes through $x$. Then $v$ is $i$-close to $x$ and $w$ is $j$-close to $x$ for some $i,j\in\{1,\dots,k-1\}$ with $i+j\leq k$. At most $\Delta^{i}$ vertices are $i$-close to $x$.
  Thus, the number of edges of $G$ that pass through $x$ is at most
  \[
  \sum_{i=1}^{k-1} \sum_{j=1}^{k-i} \Delta^i \Delta^j
  = \sum_{i=1}^{k-1} \Delta^i  \sum_{j=1}^{k-i} \Delta^j
  < \sum_{i=1}^{k-1} \Delta^i  2 \Delta^{k-i}
  = \sum_{i=1}^{k-1} 2\Delta^k
  < 2k \Delta^k \enspace.
  \]
  Fix a $(\Sigma,k)$-plane drawing of $G_0$ in a surface $\Sigma$ of Euler genus at most $g$. Choose $\epsilon>\epsilon'>0$.
  For each vertex $v$ of $G_0$, let $B_v$ be the \defin{ball} $\{x\in \mathbb{R}^2:\dist(x,v)\leq\epsilon\}$.
  For each edge $vw$ of $G_0$, let $C_{vw}$ be the \defin{channel} $\{x\in \mathbb{R}^2:\dist(x,vw)\leq\epsilon'\}\setminus(B_v\cup B_w)$.
  We may choose $\epsilon>\epsilon'>0$ sufficiently small so that
  (i) $B_v\cap B_w=\emptyset$ for distinct $v,w\in V(G_0)$,
  (ii) $C_{vw}\neq\emptyset$ for $vw\in E(G_0)$,
  (iii) $B_v\cap C_{xy}=\emptyset$ for $v\in V(G_0)$ and $xy\in E(G_0)$, and
  (iv) $C_{vw}\cap C_{xy}=\emptyset$ for distinct $vw,xy\in E(G_0)$.
  Draw each edge $vw$ of $G$ following the sequence of balls and channels defined by the vertices and edges in $P_{vw}$.
  This can be done so that whenever edges $vw$ and $xy$ of $G$ cross, the crossing point is in $B_z$ for some vertex $z$ of $G_0$ that is internal in $P_{vw}$ or in $P_{xy}$, and $vw$ and $xy$ cross at most once in each such $B_z$.

  Thus, for each edge $vw$ of $G$, every edge of $G$ that crosses $vw$ passes through a vertex on $P_{vw}$ (including $v$ and/or $w$ if they too have degree at most $\Delta$).
  Since $P_{vw}$ has at most $k+1$ vertices, and less than $2k\Delta^{k}$ edges of $G$ pass through each vertex on $P_{vw}$, the edge $vw$ is crossed by less than $2k(k+1)\Delta^{k}$ edges in $G$. Hence $G$ is $(g, 2k(k+1)\Delta^{k} )$-planar.
\end{proof}

\section{Graphs Derived from Shortcut Systems}
\label{examples}

In this section we give four examples of well-known graph families that can be expressed in terms of shortcut systems and which therefore have product structure theorems.

%%%%%%%
\subsection{Bounded-Degree String Graphs}
\label{first_example}

A \defin{string graph} is the intersection graph of a set of curves (\defin{strings}) in the plane with no three curves meeting at a single point; see  \cite{PachToth-DCG02,FP10,FP14} for example. For an integer $\delta\geq 2$, if each curve is in at most $\delta$ intersections with other curves, then the corresponding string graph is called a \defin{$\delta$-string graph}.\footnote{Every $\delta$-string graph has maximum degree at most $\delta$.  However, it is not necessarily the case that every string graph of maximum degree $\delta$ is a $\delta$-string graph.  If a pair of strings intersect in $c\ge 1$ points, then this counts as $c$ crossings for each string, but only contributes one to their degree.} A \defin{$(g,\delta)$-string} graph is defined analogously for curves on a surface of Euler genus at most $g$.

\begin{lem}
\label{StringShortcut}
Every $(g,\delta)$-string graph $G$ is contained in $G_0^\SS$ for some graph $G_0$ with Euler genus at most $g$ and some $(\delta+1,\delta+1 )$-shortcut system $\SS$ for $G_0$.
%Moreover, every internal vertex on the paths in $\SS$ has degree 4 in $G_0$.
\end{lem}

\begin{proof}
Let $\mathcal{C}=\{C_v:v\in V(G)\}$ be a set of curves in a surface of Euler genus at most $g$ whose intersection graph is $G$.  Let $G_0$ be the graph obtained by adding a vertex at the intersection point of every pair of curves in $\mathcal{C}$ that intersect,  where two such consecutive vertices on a curve $C_v$ are adjacent in $G_0$. For each vertex $v\in V(G)$, if $C_v$ intersects $k\leq\delta$ other curves, then introduce a new vertex called $v$ on $C_v$ between the
$\floor{\frac{k}{2}}$-th vertex already on $C_v$ and the $\floor{\frac{k}{2}+1}$-th such vertex. For each edge $vw$ of $G$, there is a path $P_{vw}$ of length at most $2\ceil{\frac{\delta}{2}}\leq \delta+1$ in $G_0$ between $v$ and $w$. Let $\SS$ be the set of all such paths $P_{vw}$. Consider a vertex $x$ in $G_0$ that is an internal vertex on some path in $\SS$. Then $x$ is at the intersection of $C_v$ and $C_w$ for some edge $vw\in E(G)$. If some path $P\in \SS$ passes through $x$, then $P=P_{vu}$ for some edge $vu$ incident to $v$, or $P=P_{wu}$ for some edge $wu$ incident to $w$. At most $\ceil{\frac{\delta}{2}}$ paths in $\SS$ corresponding to edges incident to $v$ pass through $x$, and similarly for edges incident to $w$. Thus at most $2\ceil{\frac{\delta}{2}}\leq\delta+1$ paths in $\SS$ use $x$ as an internal vertex. Thus $\SS$ is a $(\delta+1,\delta+1)$-shortcut system for $G_0$, and by construction, $G \subseteq G_0^\SS$.
\end{proof}

\cref{StringShortcut,ShortcutProduct,GenusProduct} now imply:

\begin{thm}\label{StringProduct}
For integers $g\geq 0$ and $\delta\geq 2$, let $\ell:= \max\{2g,3\} \,(\delta^4 + 4 \delta^3 + 9 \delta^2 + 10 \delta + 4)$ and $t:= \binom{ \delta+4}{3}-1$.
Then  every $(g,\delta)$-string graph is contained in $H\boxtimes P \boxtimes K_{\ell}$ for some path $P$ and for some graph $H$ with treewidth $t$,
\end{thm}

%%%%%%%%%%
\subsection{Powers of Bounded Degree Graphs}
\label{Powers}
\label{mid_example}

Recall that the \defin{$k$-th power} of a graph $G$ is the graph $G^k$ with vertex set $V(G^k):=V(G)$, where $vw\in E(G^k)$ if and only if $\dist_G(v,w)\leq k$. If $G$ is planar with maximum degree $\Delta$, then $G^k$ is $2k(k+1)\Delta^{k}$-planar by \cref{DrawG}.  Thus we can immediately conclude that bounded powers of planar graphs of bounded degree have product structure. However, the bounds we obtain are improved by the following lemma that constructs a shortcut system directly.

\begin{lem}
\label{PowerShortcut}
If a graph $G$ has maximum degree $\Delta$, then $G^k = G^\SS$ for some $(k,2k \Delta^{k})$-shortcut system $\SS$.
\end{lem}

\begin{proof}
For each pair of vertices $x$ and $y$ in $G$ with $\dist_G(x,y)\in\{1,\dots,k\}$, fix an $xy$-path $P_{xy}$ of length
$\dist_G(x,y)$  in $G$. Let $\SS:=\{P_{xy}: \dist_G(x,y)\in\{1,\dots,k\} \}$. Say $P_{xy}$ uses some vertex $v$ as an internal vertex. If $\dist_G(v,x)=i$ and $\dist_G(v,y)=j$, then $i,j\in\{1,\dots,k-1\}$ and $i+j\leq k$. The number of vertices at distance $i$ from $v$ is at most $\Delta^i$. Thus the number of paths in $\SS$ that use $v$ as an internal vertex is at most
\[\sum_{i=1}^{k-1} \sum_{j=1}^{k-i} \Delta^i\Delta^j
= \sum_{i=1}^{k-1} \Delta^i \sum_{j=1}^{k-i} \Delta^j
< \sum_{i=1}^{k-1} \Delta^i ( 2 \Delta^{k-i} )
< 2k \Delta^k\enspace.\]
Hence $\SS$ is a $(k, 2k \Delta^k)$-shortcut system.
\end{proof}

\cref{ShortcutProduct,PowerShortcut} imply:

\begin{thm}
\label{PowerProduct}
Let $G$ be a graph of maximum degree $\Delta$ that is contained in $H\boxtimes P\boxtimes K_\ell$, for some graph $H$ of treewidth at most $t$ and for some path $P$. Then for every integer $k\geq 1$, the $k$-th power $G^k$ is contained in $J\boxtimes P\boxtimes K_{2k \ell \Delta^{k}(k^3+3k)}$ for some graph $J$ of treewidth at most $\binom{k+t}{t}-1$ and some path $P$.
\end{thm}

\subsection{Map Graphs of Bounded Clique Number}

Map graphs are defined as follows. Start with a graph $G_0$ embedded in a surface of Euler genus $g$, with each face labelled a `nation' or a `lake', where each vertex of $G_0$ is incident with at most $d$ nations. Let $G$ be the graph whose vertices are the nations of $G_0$, where two vertices are adjacent in $G$ if the corresponding faces in $G_0$ share a vertex. Then $G$ is called a \defin{$(g,d)$-map graph}.  A $(0,d)$-map graph is called a (plane) \defin{$d$-map graph}; see \citep{FLS-SODA12,CGP02} for example. The $(g,3)$-map graphs are precisely the graphs of Euler genus at most $g$; see \citep{dujmovic.eppstein.ea:structure}.

There is a natural drawing of a map graph obtained by positioning each vertex of $G$ inside the corresponding nation and each edge of $G$ as a curve passing through the corresponding vertex of $G_0$. It is easily seen that each edge of this drawing is involved in at most $\floor{\frac{d-2}{2}}\ceil{\frac{d-2}{2}}$ crossings; see \citep{dujmovic.eppstein.ea:structure}. Thus $G$ is $(g,\floor{\frac{d-2}{2}}\ceil{\frac{d-2}{2}})$-planar. Also note that \cref{DrawG} with $k=2$ implies that $G$ is $(g, O(d^{2}) )$-planar. \cref{gkPlanarProduct} then establishes a product structure theorem for map graphs, but we get much better bounds by constructing a shortcut system directly.  The following lemma is reminiscent of the characterisation of $(g,d)$-map graphs in terms of the half-square of bipartite graphs \citep{CGP02,dujmovic.eppstein.ea:structure}.

\begin{lem}
\label{MapShortcut}
Every $(g,d)$-map graph $G$ is contained in $G_1^\SS$ for some graph $G_1$ with Euler genus at most $g$ and some $(2,\tfrac12 d(d-3) )$-shortcut system $\SS$ for $G_1$.
\end{lem}

\begin{proof}
Let $G$ be a $(g,d)$-map graph. So there is a graph $G_0$ embedded in a surface of Euler genus $g$, with each face labelled a `nation' or a `lake', where each vertex of $G_0$ is incident with at most $d$ nations. Let $N$ be the set of nations. Then $V(G)=N$ where two vertices are adjacent in $G$ if the corresponding nation faces of $G_0$ share a vertex. Let $G_1$ be the graph with $V(G_1):=V(G_0) \cup N$, where distinct vertices $v,w\in N$ are adjacent in $G_1$ if the boundaries of the corresponding nations have an edge of $G_0$ in common, and $v\in V(G_0)$ and $w\in N$ are adjacent in $G_1$ if $v$ is on the boundary of the nation corresponding to $w$. Observe that $G_1$ embeds in the same surface as $G_0$ with no crossings, and that each vertex in $V(G_0)$ has degree at most $d$ in $G_1$. Consider an edge $vw\in E(G)$. If the nations corresponding to $v$ and $w$ share an edge of $G_0$, then $vw$ is an edge of $G_1$. Otherwise,  $v$ and $w$ have a common neighbour $x$ in $V(G_0)$. In the latter case, let $P_{vw}$ be the path $(v,x,w)$. Let $\SS$ be the set of all such paths $P_{vw}$. Each vertex $x\in V(G_0)$ is the middle vertex on at most $\tfrac12 d(d-3)$  paths in $\SS$. Thus $\SS$ is a $(2,\tfrac12 d(d-3))$-shortcut system for $G_1$, and by construction, $G \subseteq G_1^\SS$.
\end{proof}

\cref{ShortcutProduct,MapShortcut,GenusProduct} imply:

\begin{thm}
\label{gdMapProduct}
For integers $g\geq 0$ and $d\geq 3$, let $\ell:= \max\{2g,3\}(7d^2-21d)$.
Then every $(g,d)$-map graph is contained in $H\boxtimes P \boxtimes K_{\ell}$ for some path $P$ and for some graph $H$ with treewidth $9$,
\end{thm}

\subsection{$k$-Nearest-Neighbour Graphs}
\label{last_example}

In this section, we show that $k$-nearest neighbour graphs of point sets in the plane are $O(k^2)$-planar.  For two points $x,y\in\R^2$, let $d_2(x,y)$ denote the Euclidean distance between $x$ and $y$. The $k$-nearest-neighbour graph of a point set $P\subset\R^2$ is the geometric graph $G$ with vertex set $V(G)=P$, where the edge set is defined as follows. For each point $v\in P$, let $N_k(v)$ be the set of $k$ points in $P$ closest to $v$. Then $vw\in E(G)$ if and only if $w\in N_k(v)$ or $v\in N_k(w)$. (The edges of $G$ are straight-line segments joining their endpoints.) See \citep{ProximityGraphs} for a survey of results on $k$-nearest neighbour graphs and other related proximity graphs.

The following result, which is immediate from \citet[Corollary~4.2.6]{abrego.munroy.ea:on} states that $k$-nearest-neighbour graphs have bounded maximum degree:
\begin{lem}
\label{k-nn-max-degree}
The degree of every vertex in a $k$-nearest-neighbour graph is at most $6k$.
\end{lem}

We make use of the following well-known observation (see for example, \citet[Lemma~2]{bose.morin.ea:routing}):
\begin{obs}
\label{convex}
If $v_0,\ldots,v_3$ are the vertices of a convex quadrilateral in counterclockwise order then there exists at least one $i\in\{0,\ldots,3\}$ such that $\max\{d_2(v_i,v_{i-1}), d_2(v_i,v_{i+1})\} < d_2(v_{i-1},v_{i+1})$, where subscripts are taken modulo 4.
\end{obs}

\begin{lem}
\label{nearest-neighbour}
  Every $k$-nearest-neighbour graph is $O(k^2)$-planar.
\end{lem}

\begin{proof}
  Let $G$ be a $k$-nearest-neighbour graph and consider any edge $vw\in E(G)$.
  Let $xy\in E(G)$ be an edge that crosses $vw$.  Note that $vxwy$ are the vertices of a convex quadrilateral in (without loss of generality) counter-clockwise order. Then we say that
  \begin{compactenum}
    \item $xy$ is of Type~$v$ if $\max\{d_2(v,x), d_2(v,y)\}< d_2(x,y)$;
    \item $xy$ is of Type~$w$ if $\max\{d_2(w,x), d_2(w,y)\}< d_2(x,y)$; or
    \item $xy$ is of Type~C otherwise.
  \end{compactenum}
  If $xy$ is of Type~C, then \cref{convex} implies that $\max\{d_2(x,v),d_2(x,w)\} < d_2(v,w)$ without loss of generality.
   In this case, we call $x$ a Type~C vertex.  We claim that $V(G)$ contains at most $k-1$ Type~C vertices.  Indeed, more than $k-1$ Type~C vertices would contradict the fact that $vw\in E(G)$ since every Type~C vertex is closer to both $v$ and $w$ than $d_2(v,w)$.

  Next observe that, if $xy$ is of Type~$v$, then at least one of $xv$ or $yv$ is in $E(G)$ in which case we call $x$ (respectively $y$) a Type~$v$ vertex.  By \cref{k-nn-max-degree}, there are at most $6k$ Type~$v$ vertices.  Similarly, there are at most $6k$ Type~$w$ vertices.

  Thus, in total, there are at most $13k-1$ Type~$v$, Type~$w$, and Type~C vertices. By \cref{k-nn-max-degree}, each of these vertices is incident with at most $6k$ edges that cross $vw$. Therefore, there are at most $78k^2-6k$ edges of $G$ that cross $vw$.  Since this is true for every edge $vw\in E(G)$, $G$ is $(78k^2-6k)$-planar.
\end{proof}

Note that \cref{nearest-neighbour} is tight up to the leading constant:  Every $k$-nearest neighbour graph on $n\ge k+1$ vertices has at least $kn/2$ edges and at most $kn$ edges.  For $k\ge 7$, the Crossing Lemma~\citep{ajtai.chvatal.ea:crossing-free,leighton:complexity} implies that the total number of crossings is therefore $\Omega(k^3n)$ so that the average number of crossings per edge is $\Omega(k^2)$.

\cref{nearest-neighbour} and \cref{kPlanarProduct} immediately imply

\begin{thm}\label{nn_product_structure}
  Every $k$-nearest neighbour graph is a subgraph of $H\boxtimes P\boxtimes K_{\ell}$ for some graph $H$ of treewidth $O(k^6)$ and for some $\ell\in O(k^4)$.
\end{thm}

\section{Framed Graphs}
\label{FramedSection}

For any integer $d\ge 3$, an $\R^2$-embedded (multi)graph $G$ is \defin{$d$-framed} if it has a biconnected plane spanning (multi)subgraph $G_0$, whose facial cycles each have length at most $d$,  and such that the interior of each edge in $E(G)\setminus E(G_0)$ is contained in some face (possibly the outer face) of $G_0$. The embedded (multi)graph $G_0$ is called the \defin{frame} of $G$. This definition (for simple graphs) was introduced by \citet{BDGGMR}.  A $d$-framed graph $G$ is \defin{edge-maximal} with respect to its frame $G_0$ if, for each face $F$ of $G_0$ and each pair of distinct vertices $v,w\in V(F)$, $vw\in E(F)$ or $G$ contains an edge $vw$ whose interior is contained in $F$.

Every edge-maximal $d$-framed multigraph can be described by a $(2,d(d-3)/2)$-shortcut system applied to a plane multigraph by adding a vertex inside each face $F$ of $G_0$ adjacent to the vertices of $F$ and creating a length-$2$ shortcut between each non-adjacent pair of vertices in $V(F)$. Thus \cref{ShortcutProduct,PlanarProduct} imply that any $d$-framed multigraph $G$ is a subgraph of $H\boxtimes P\boxtimes K_{O(d^2)}$ where the treewidth of $H$ is at most $10$.  However, for framed graphs we prove the following stronger result, where the treewidth bound on $H$ is best possible and the dependence on $d$ is only linear.\footnote{A multigraph $G$ is \defin{contained} in a graph $X$ if the simple graph underlying $G$ is contained in $X$.}

\begin{thm}\label{d_framed_product_stucture}
For any integer $d\ge 3$, every $d$-framed multigraph is contained in $H\boxtimes P\boxtimes K_{d+3\floor{d/2}-3}$ for some planar graph $H$ with treewidth at most 3 and for some path $P$.
\end{thm}

\cref{d_framed_product_stucture} is used below to obtain product structure theorems for 1-planar multigraphs and for map graphs.  Note that even for simple 1-planar graphs it is essential that we allow frames with parallel edges.

The following definition is a convenient way to establish the planarity of $H$ in \cref{d_framed_product_stucture}:  A $T$-decomposition of a graph $G$ is \defin{$t$-simple} if each of its bags has size at most $t+1$ and, for each $t$-element subset $\{v_1,\ldots,v_t\}\subseteq V(G)$, at most two bags contain $v_1,\ldots,v_t$.  To show the planarity of $H$ in \cref{d_framed_product_stucture} we use the fact that a graph has a 3-simple tree-decomposition if and only if it has treewidth at most 3 and is planar~\citep{knauer.ueckerdt:simple,kratochvil.vaner:note}.

% The following technical lemma is analogous to Lemma~17 in \cite{dujmovic.joret.ea:planar}, which performs the inductive step in the construction of a so-called \emph{tripod decomposition}.
%
%
%
%
%
%  cycle $F$ in a near-triangulation $G$ and a partition of $V(F)$ into at most three bipods,
%
%
% to how a near-triangulation bounded by a cycle $F$ in in a triangulation $G$ whose vertices have already been assigned
%
% we offer some comparison.  \cite[Lemma~17]{dujmovic.joret.ea:planar} shows that, given a partial partition of $V(G)$ into tripods, and a cycle $F$ in $G$ whose vert

A \defin{BFS spanning forest} $T$ of $G$ rooted at a set $V_0\subseteq V(G)$ is a spanning forest of $G$ that contains a tree rooted at each vertex of $V_0$ with the property that for each $v\in V(G)$,
\[  \min\{\dist_G(v,w):w\in V_0\}=\min\{\dist_T(v,w):w\in V_0\}.  \]
\cref{d_framed_product_stucture} is a consequence of the following technical lemma, which is an extension of the analogous result for plane graphs~\cite[Lemma~17]{dujmovic.joret.ea:planar}.  For readers already familiar with the proof in \cite{dujmovic.joret.ea:planar}, the main difference here is that a typical tripod in the proof of \cite[Lemma~17]{dujmovic.joret.ea:planar} consists of three upward paths in a BFS tree/forest whose lower endpoints belong to a single triangular face $\tau$.  In the current setting, a typical tripod consists of three upward paths whose lower endpoints belong to a common face $\tau$ of $G_0$ plus all of the (at most $d$) vertices of $\tau$ that are not already included in tripods.

\begin{lem}
  \label{induction} The setup:
  \begin{compactenum}
    \item Let $G$ be a $d$-framed multigraph and let $G_0$ be a frame of $G$.
    \item Let $F_0$ be the outer face of $G_0$ and let $T$ be a BFS spanning forest of $G_0$ rooted at $V(F_0)$.
    \item For every integer $j\ge 0$, let $L_j=\{v\in V(G):\dist_T(v,V(F_0))=j\}$.
    \item Let $F$ be a cycle in $G_0$ whose vertices are partitioned into $k\le 3$ non-empty sets $P_1,\ldots,P_k$ such that for each $i\in\{1,\ldots,k\}$,
    \begin{compactenum}
      \item $F[P_i]$ is connected; and
      \item $P_i$ has a partition $\{X_i,Y_i\}$ where $|X_i|\le d-3$ and $|Y_i\cap L_j| \le 3$ for each $j\ge 0$.
    \end{compactenum}
    \item Let $N$ and $N_0$ be the subgraphs of $G$ and $G_0$ consisting only of those edges and vertices contained in the closure of the interior of the Jordan curve defined by $F$.
  \end{compactenum}
  Then $N$ has an $H$-partition $\PP=\{S_x: x\in V(H)\}$ such that:
  \begin{compactenum}[(i)]
    \item for each $i\in\{1,\ldots,k\}$, there exists $x_i\in V(H)$ such that $P_i=S_{x_i}$;
    \item For each $x\in V(H)$, $S_x$ has a partition $\{X_x,Y_x\}$ where $|X_x|\le d-3$ and $|Y_x\cap L_j|\le 3$ for each $j\ge 0$;
    \item $H$ has a $3$-simple tree-decomposition $\mathcal{T}$ and, if $k=3$, then $\{x_1,x_2,x_3\}$ is contained in exactly one bag of $\mathcal{T}$.
  \end{compactenum}
\end{lem}

\begin{proof}
This proof is similar to the proofs of Lemmas~13 and 17 in \cite{dujmovic.joret.ea:planar}, but is complicated by several factors.  In particular, the reader should keep in mind that $G$ and $G_0$ are multigraphs and that the cycle $F$ may consist of two vertices and two (parallel) edges.

We proceed by induction on $(x,y)$ where $x$ is the number of inner vertices of $N_0$ and $y$ is the number of inner faces of $N_0$.  The induction base case occurs when $N_0$ has no inner vertices (that is, $V(N_0)=V(F)$). In this case, simply take $\PP:=\{P_1,\ldots,P_k\}$ and verify that the preconditions of the lemma (the `setup') ensure that $\PP$ satisfies the requirements (i)--(iii):
  \begin{compactenum}[(i)]
    \item By definition, each of $P_1,\ldots,P_k$ is a part of $\mathcal{P}$.
    \item Assumption 4(b) ensures, for each $x\in V(H)$, the existence of $X_x$ and $Y_x$ such that $|X_x|\le d-3$ and $|Y_x\cap L_j|\le 3$ for each integer $j\ge 0$.
    \item The trivial tree-decomposition of $H$ that has a single bag containing $\{x_1,\ldots,x_k\}$ is $3$-simple and, if $k=3$, has exactly one bag that contains $\{x_1,\ldots,x_3\}$.
  \end{compactenum}

We now move onto the case in which $N_0$ contains at least one inner vertex. If $N_0$ contains a face bounded by two parallel edge $e$ and $e'$ with $e\in E(F)$, then we can apply the inductive hypothesis on the cycle $F'$ obtained from $F$ by replacing $e$ with $e'$. (This is a valid application of induction since $F'$ has fewer faces of $G_0$ in its interior than $F$ does.)  In this way, we may assume that each edge of $F$ bounds a face of $N_0$ that contains at least three vertices.

Our $H$-partition $\mathcal{P}$ will include $P_1,\dots,P_k$ and a (possibly empty) part $S$ that will be determined by three vertices $v_1v_2v_3$ that belong to a common inner face $\tau$ of $N_0$.  We first explain how $v_1$, $v_2$ and $v_3$ are chosen and then explain how these are used to define $S$.  There is one case to consider for each possible value of $k$ (see \cref{boring_figure}):

\begin{compactenum}
    \item If $k=1$ then let $\tau:=v_1,\ldots,v_p$ be (the facial cycle of) an inner face of $N_0$ that contains at least one edge $v_1v_2$ of $F$ on its boundary.

    \item If $k= 2$ then let $\tau:=v_1,\ldots,v_p$ be (the facial cycle of) an inner face of $N_0$ that contains at least one edge $v_1v_2$ of $F$ on its boundary and such that $v_1\in P_1$ and $v_2\in P_2$.

    \item If $k=3$ then assign each vertex $v$ of $N_0$ a colour $\alpha(v)\in\{1,\ldots,3\}$ as follows:  Let $w$ be the first vertex of $V(F)$ encountered on the path in $T$ from $v$ to the root of the tree that contains $v$.  (This implies that $w=v$ when $v$ is a vertex of $F$.)  Since $w\in V(F)$, $w\in P_i$ for some $i\in\{1,\ldots,k\}$ and we define $\alpha(v):=i$.  Now, for each $d'\in\{4,\ldots,d\}$ and each $d'$-sided inner face $F$ of $N_0$, add $d'-3$ edges inside of $F$ to split it into $3$-sided faces.  This yields a plane supergraph $N_0'$ of $N_0$ in which each inner face has exactly three edges on its boundary. By Sperner's Lemma (see \citep{Proofs4}) there exists an inner face $\tau':=v_1,v_2,v_3$ of $N_0'$ such that $\alpha(v_i)=i$ for each $i\in\{1,2,3\}$.\footnote{Although Sperner's Lemma is usually stated for simple graphs, the proof considers a certain graph defined over the faces of $N_0'$.  The proof relies only on the fact that the number of odd-degree vertices in this graph is even (known as the \defin{Handshake Lemma}), which is true even when this graph is non-simple.}  Let $\tau:=v_1,\ldots,v_p$ be (the facial cycle of) the face in $N_0$ that contains $\tau'$.
  \end{compactenum}

  \begin{figure}[!h]
    \begin{center}
      \begin{tabular}{c@{}c@{}c}
        \includegraphics{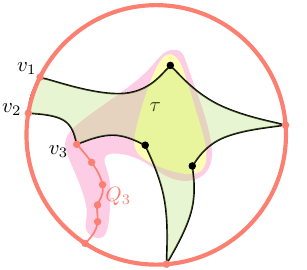} &
        \includegraphics{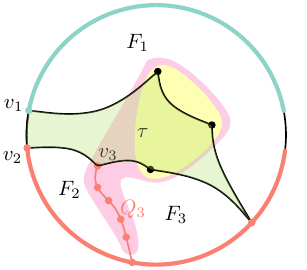} &
        \includegraphics{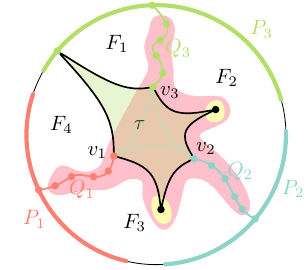} \\
        \multicolumn{3}{c}{{\color{brew8}\raisebox{-3pt}{\rule{12pt}{12pt}}} $S$\qquad {\color{brew2}\raisebox{-3pt}{\rule{12pt}{12pt}}} $X$} \\
        $k=1$ & $k=2$ & $k=3$
      \end{tabular}
    \end{center}
    \caption{Three cases when choosing the face $\tau$ of $N_0$.}
    \label{boring_figure}
  \end{figure}

  We now define the vertices in $S$.  For each $i\in\{1,2,3\}$, let $Q_i$ be the shortest path, in $T$, from $v_i$ to a vertex in $V(F)$.  Let $\overline{S}$ denote the subgraph of $N_0$ consisting of vertices and edges of $Q_1$, $Q_2$, $Q_3$, and $\tau$. Let $S:=V(\overline{S})\setminus V(F)$, let $Y:=(V(Q_1)\cup V(Q_2)\cup V(Q_3))\setminus V(F)$, and let $X:=S\setminus Y$. Observe that $|V(Q_i)\cap L_j|\le 1$   for each $i\in\{1,2,3\}$ and each $j\ge 0$.  Therefore $|Y\cap L_j|\le 3$ for each $j\ge 0$. Furthermore, $X\subseteq V(\tau)\setminus\{v_1,v_2,v_3\}$, so $|X|\le p-3\le d-3$. Therefore, using $S:=Y\cup X$ as a part in the partition $\PP$ satisfies condition~(ii).

  Let $M$ denote the subgraph of $N_0$ containing the edges and vertices of $\overline{S}$ and the edges and vertices of $F$. By our choice of $\tau$, the graph $M$ is $2$-connected. Let $F_1,\ldots,F_m,\tau$ be the inner faces of $M$. For each $i\in\{1,\ldots,m\}$, the choice of $v_1v_2v_3$ ensures that $V(F_i)$ can be partitioned into at most three sets $P_{i,1}$, $P_{i,2}$, and $P_{i,3}$ where $P_{i,1}\subseteq S$, $P_{i,2}\subseteq P_b$, and $P_{i,3}\subseteq P_c$ for some $b,c\in\{1,\ldots,k\}$.

  In the following, in order to avoid introducing even more notation or constantly adding the modifier ``when it exists'', we assume that $V(F_i)\cap S\neq\emptyset$ and that there are exactly two distinct $b,c\in\{1,\ldots,k\}$ such that $V(F_i)\cap P_b\neq\emptyset$ and $V(F_i)\cap P_c\neq\emptyset$.  When this assumption does not hold, exactly the same reasoning can be used, while omitting one or more of $P_{i,1}$, $P_{i,2}$ or $P_{i,3}$.

  Since $M$ is $2$-connected and $\overline{S}$ is connected, $F_i[P_{i,j}]$ is connected, for each $j\in\{1,2,3\}$.  Let $N_i$ and $N_{0,i}$ be the subgraphs of $N$ and $N_0$ contained in the closure of the interior of $F_i$. Every inner vertex of $N_{0,i}$ is an inner vertex of $N_{0}$.  Every inner face of $N_{0,i}$ is an inner face of $N_0$ and $\tau$ is not an inner face of $N_{0,i}$.  Therefore $N_{0,i}$ has no more inner vertices than $N_0$ and has fewer inner faces than $N_{0}$.  Therefore, we can apply induction using the cycle $F_i$ and the partition $P_{i,1},P_{i,2},P_{i,3}$ of $V(F_i)$.  The result of the induction is a $H_i$-partition $\mathcal{P}_i$ of $N_i$ satisfying (i)--(iii) above. We now define our partition
  \[
    \mathcal{P}:=\{Y, P_1,\ldots,P_k\} \cup \bigcup_{i=1}^m \{P\in\mathcal{P}_i: P\cap V(M)=\emptyset\} \enspace .
  \]
  It is straightforward to verify that $\mathcal{P}$ is indeed a partition of $V(N)$ and, by definition this is an $H$-partition of $N$ for the graph $H:=N/\mathcal{P}$.  It remains to verify that $\mathcal{P}$ satisfies (i)--(iii).

  By construction $\mathcal{P}$ contains parts $P_1,\ldots,P_k$, so $\mathcal{P}$ satisfies (i).  By Assumption~4(b) each of the parts $P_1,P_2,P_3$ satisfies condition~(ii).  We have already argued that the part $S$ satisfies condition~(ii).  The inductive hypothesis ensures that each of the remaining parts satisfies condition~(ii).   Therefore $\mathcal{P}$ satisfies (ii).

  It remains to construct a tree-decomposition $\mathcal{T}:=(B_y:y\in V(T_0))$ of $H$ that satisfies (iii). Let $x_0$ be the vertex of $H$ obtained from contracting $S$ and, for each $j\in\{1,\ldots,k\}$, let $x_j$ denote the vertex of $H$ obtained from contracting $P_j$.  The tree $T_0$ has a node $y_0$ with bag $B_{y_0}:=\{x_0,\ldots,x_k\}$. The inductive hypothesis ensures that, for each $i\in\{1,\ldots,m\}$,  $H_i:=N_i/\mathcal{P}_i$ has a tree-decomposition $\mathcal{T}_i:=(B_z:z\in V(T_i))$ that satisfies (iii).  Let $a_i,b_i,c_i\in V(H_i)$ be the vertices obtained by contracting $P_{i,1}$, $P_{i,2}$, and $P_{i,3}$, respectively. Since $a_i$, $b_i$, $c_i$ form a clique in $H_i$, there exists a bag $B_{z_i}$ in $\mathcal{T}_i$ that contains $\{a_i,b_i,c_i\}$.

  Recall that $P_{i,1}:=V(N_i)\cap S$, $P_{i,2}:=V(N_i)\cap P_b$ and $P_{i,3}:=V(N_i)\cap P_c$ for some $b,c\in\{1,\ldots,k\}$.  For each bag $B_z$, $z\in V(T_i)$, replace each occurrence of $a_i$ by $x_0$, replace each occurrence of $b_i$ by $x_b$ and replace each occurrence of $c_i$ by $x_c$.  These replacements do not increase the size of any bag and they result in a tree-decomposition of the graph obtained from $H_i$ by renaming the vertices $a_i$, $b_i$, and $c_i$ with the names $x_0$, $x_a$, and $x_b$. Now add an edge that joins $y_0$ to $z_i$ so that $T_i$ becomes a subtree of $T_0$ that is adjacent to $y_0$.  We perform this operation for each $i\in\{1,\ldots,m\}$ to obtain our final tree-decomposition $\mathcal{T}:=(B_y:y\in V(T_0))$.

  It is now straightforward to verify that $\mathcal{T}$ is a tree-decomposition of $H_0:=N_0/\mathcal{P}$ in which each bag has size at most $4$.    To see that $\mathcal{T}$ is also a tree-decomposition of $H:=N/\mathcal{P}$, consider any edge $vw\in E(N)\setminus E(N_0)$ with $v\in P_\alpha$ and $w\in P_\beta$.  We must verify that some bag of $\mathcal{T}$ contains $\alpha$ and $\beta$.  If both $v$ and $w$ are in $V(M)$, then $\alpha=x_a$ and $\beta=x_b$ for some $a,b\in\{0,\ldots,3\}$, in which case $B_{y_0}=\{x_0,\ldots,x_k\}$ contains $\alpha$ and $\beta$.  If neither $v$ nor $w$ are in $V(M)$ then, since no edge of $N\subseteq G$ crosses an edge of $M\subseteq G_0$, both $v$ and $w$ are contained in the same face $F_i$ of $M$. So, by the inductive hypothesis there exists a bag of $\mathcal{T}_i$ that contains $\alpha$ and $\beta$ and this becomes a bag of $\mathcal{T}$ that contains $\alpha$ and $\beta$.  Finally, if $v\in V(M)$ and $w\not\in V(M)$, then $v\in P_{i,j}$ for some $j\in\{1,2,3\}$ and $w$ is in some face $F_i$ of $M$. By induction, some bag of $\mathcal{T}_i$ contains $\beta$ and $\alpha'$, where $\alpha'$ is the vertex of $H_i$ obtained by contracting $P_{i,j}$.  However, before $T_i$ is attached to $T_0$ each occurence of $\alpha'$ in $\mathcal{T}_i$ is replaced by $\alpha$, so some bag of $\mathcal{T}$ contains both $\alpha$ and $\beta$. Therefore $\mathcal{T}$ is a tree-decomposition of $H$.

  Next we verify that, when $k=3$ exactly one bag of $\mathcal{T}$ contains $\{x_1,\ldots,x_3\}$.  The bag $B_{y_0}$ contains $x_1,\ldots,x_k$. When $k=3$, each neighbour of $z_i$ of $y_0$ has a bag that contains at most two of $\{x_1,x_2,x_3\}$.  Therefore, when $k=3$, $B_y$ is the unique bag in $\mathcal{T}$ that contains $\{x_1,x_2,x_3\}$.

  It remains to show that any triple of vertices of $H$ is contained in at most two bags of $\mathcal{T}$. By the inductive hypothesis, the only triples we need to be concerned about are the $3$-element subsets of  $\{x_0,x_1,\ldots,x_k\}$. The case $k=1$ is trivial since no triples can be formed by $\{x_0,x_1\}$.  When $k=2$, $B_{y_0}=\{x_0,x_1,x_2\}$.  In this case, the fact that $\tau$ contains an edge $v_1v_2$ of $F$ with $v_1\in P_{x_1}$ and $v_2\in P_{x_2}$ ensures that there is at most one face $F_i$ of $M$ that has vertices of $S$, $P_1$, and $P_2$ on its boundary  (see the middle part of \cref{boring_figure}).  By the inductive hypothesis (applied to $H_i$) the bag $B_{z_i}$ is the only bag of $\mathcal{T}_i$ that contains $\{x_0,x_1,x_2\}$.  Therefore there are at most two bags, $B_y$ and $B_{z_i}$ of $\mathcal{T}$ that contain $\{x_0,x_1,x_2\}$.  This establishes (iii) for $k\in\{1,2\}$ so, from this point on we assume that $k=3$.

  We claim that, for distinct $a,b\in \{1,2,3\}$ there is at most one face $F_i$ such that $V(F_i)\cap S_{x_j}\neq\emptyset$ for each $j\in\{0,a,b\}$.  Suppose to the contrary that there are two such faces $F_{i_1}$ and $F_{i_2}$.  Create a graph $M^+$ from $M$ by placing a vertex $v_{i_b}$ inside $F_{i_b}$ and adjacent to each vertex in $V(F_{i_b})$ for each $b\in\{1,2\}$.  Since $M$ is planar and $F_{i_1}$ and $F_{i_2}$ are distinct faces of $M$, $M^+$ is planar.  However, contracting each of $S_{x_0},\ldots,S_{x_3}$ in $M^+$ produces a graph isomorphic to the complete bipartite graph $K_{3,3}$, a contradiction.  (In this contracted graph, $x_0,x_a,x_b$ are on one side of the bipartition and $v_{i_a}$, $v_{i_b}$, and the node in $\{x_1,x_2,x_3\}\setminus\{x_a,x_b\}$ is on the other side.)

  We have already established that, when $k=3$, the only bag that contains $\{x_1,x_2,x_3\}$ is $B_{y_0}$, so we need only consider triples $\{x_0,x_a,x_b\}$ for distinct $a,b\in\{1,2,3\}$. By the preceding claim, there is at most one face $F_i$ with vertices of $S$, $P_a$, and $P_b$ on its boundary.  By the inductive hypothesis, the only other bag in $\mathcal{T}$ that contains $\{x_0,x_a,x_b\}$ is the bag $B_{z_i}$.  This completes the proof.
\end{proof}

Using \cref{induction}, the proof of \cref{d_framed_product_stucture} is now straightforward.

\begin{proof}[Proof of \cref{d_framed_product_stucture}]
Let $G$ be a $d$-framed multigraph with frame $G_0$ having outer face $F_0$
Let $T$ be a BFS forest of $G_0$ rooted at $V(F_0)$.  For each integer $j\ge 0$, let $L_j:=\{w\in V(G_0):\dist_G(w,V_0)=j\}$, so that $\langle L_0,L_1,\ldots\rangle$ is a layering of $G_0$.  Let $P_1:= V(F_0)$, let $Y_1$ be any subset of $\min\{|P_1|,3\}$ vertices in $P_1$ and let $X_1:=P_1\setminus Y_1$.  Then $G$, $G_0$, $T$, $F=F_0$, $k=1$, and $P_1=Y_1\cup X_1$ satisfy the conditions of \cref{induction}.  Apply \cref{induction} to obtain an $H$-partition $\mathcal{P}:=\{S_x:x\in V(H)\}$ of $G$ that satisfies Conditions~(i), (ii), and (iii).  Since any graph having a $3$-simple tree-decomposition is planar \cite{knauer.ueckerdt:simple,kratochvil.vaner:note,elmallah.colbourn:on}, Condition~(iii) implies that $H$ is planar.

Define $\mathcal{L}=\langle L_0',L_1'\ldots\rangle$ where $L_i'=L_{\lfloor d/2\rfloor i}\cup \cdots \cup L_{\lfloor d/2\rfloor(i+1)-1}$ for each integer $i\ge 0$. This is a layering of $G$ since $\dist_{G_0}(v,w)\le \floor{ \frac{d}{2} }$ for each edge $vw\in E(G)$.  By Condition~(ii), $|L_j\cap Y_x|\le 3$ for each integer $j\ge 0$, so $|L_i'\cap Y_x|\le 3\floor{\tfrac{d}{2}}$. Therefore, $|L_i'\cap S_x|= |L_i'\cap Y_x|+ |L_i'\cap X_x|\le 3\floor{ \tfrac{d}{2}} + d -3$ for each integer $i\ge 0$ and each $x\in V(H)$. The result now follows from \cref{PartitionProduct}.
\end{proof}

\subsection{\boldmath $1$-Planar Graphs}
\label{sec-1-planar}

\cref{1_planar_product} is an immediate consequence of \cref{d_framed_product_stucture} and the next lemma. Similar results appear in earlier works \cite{CGP06,BDGGMR,Brandenburg19,Brandenburg20} but we state the precise lemma and provide a proof here for the sake of completeness.

\begin{lem}\label{1_planar_is_4_framed}
A multigraph $G$ is $1$-planar if and only if it is contained in a $4$-framed multigraph.
\end{lem}

\begin{proof}
Let $G$ be a (connected) $1$-plane multigraph. We may assume that no two edges incident to a common vertex of $G$ cross, since such a crossing can be removed by a local modification to obtain an isomorphic 1-plane graph in which the two edges do not cross\footnote{While this is true for 1-plane graphs it is not true for $k$-plane graphs with $k\ge 3$; the uncrossing operation can increase the number of crossings on a particular edge from $k$ to $2(k-1)$.}. We make $G$ into an \defin{edge-maximal} $1$-plane multigraph, by repeating the following operation:  If distinct vertices $v$ and $w$ of $G$ appear on a common face $F$ and there is no edge $vw\in E(G)$ contained in the boundary of $F$, then add the edge $vw$, embedded in the face $F$.  This may introduce parallel edges into $G$, but no newly-introduced edge $vw$ is on the boundary of a face with an edge parallel to $vw$.  Since $G$ is connected, each edge added during this process splits a face $F$ with a facial walk of length at least $4$ into two faces each having a shorter facial walk than that of $F$. Therefore this process eventually reaches a stage where each face is either bounded by a $3$-cycle or a (pre-existing) $2$-cycle, and the process ends.

To understand the structure of $G$, it is helpful to consider the plane graph $G'$  obtained by replacing each pair of edges $vw$ and $xy$ that cross at $p$ with four edges $vp$, $wp$, $xp$, and $yp$ meeting at a newly added \defin{dummy vertex} $p$. Let $F$ be a face of $G'$ and let $W:=w_0,\ldots,w_{r-1}$ be the facial walk around $F$. In the following paragraphs, subscripts on the vertices in $W$ are implicitly taken modulo $r$. By edge-maximality, if $W$ contains only non-dummy vertices, then it contains exactly three vertices and $F$ is bounded by three edges of $G$. If $W$ contains a dummy vertex $w_i$, then neither $w_{i-1}$ nor $w_{i+1}$ is a dummy vertex since if $w_{i-1}$ (respectively $w_{i+1}$) were a dummy vertex then the edge of $G$ that contains $w_{i-1}w_i$ (respectively $w_iw_{i+1}$) would be involved in at least two crossings.  Therefore each dummy vertex $w_i$ of $W$ is flanked by two non-dummy vertices $w_{i-1},w_{i+1}\in V(G)$.  Since $w_i$ has degree $4>1$ in $G'$ and no two edges incident to a common vertex cross each other, $w_{i-1}\neq w_{i+1}$.  Since $G$ is edge-maximal, the edge $w_{i-1}w_{i+1}$ is an edge of $G$ on the boundary of $F$.  Therefore, if $F$ has a dummy vertex $w_i$ in its facial walk $W$, then $W=w_{i-1},w_i,w_{i+1}$ and $F$ is bounded by three edges of $G'$, each of which is contained in a different edge of $G$.

\begin{figure}[!h]
	\begin{center}
		\begin{tabular}{c@{\hspace{1cm}}c}
			\includegraphics{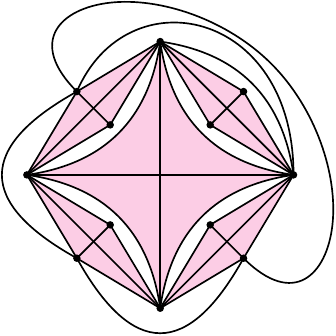} &
			\includegraphics{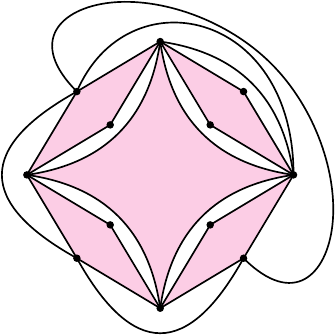} \\
			$G$ & $G_0$
		\end{tabular}
	\end{center}
	\caption{The graph $G_0$ obtained by removing pairs of crossing edges from $G$ is a plane multigraph whose faces all have three or four sides.}
	\label{one_planar_example}
\end{figure}

Refer to \cref{one_planar_example}.  Consider some pair of edges $vw$ and $xy$ that cross at point $p$, which is a degree-4 vertex of $G'$.  Since no pair of edges incident to a common vertex cross each other, $v$, $w$, $x$, and $y$ are all distinct.  There are four faces $F_1,\ldots,F_4$ of $G'$ with $p$ on their boundary and, from the preceding discussion each of $vx$, $xw$, $wy$, and $yv$ is an uncrossed edge of $G$.  Therefore, by removing each pair of crossing edges from $G$ we obtain a plane multigraph $G_0$ each of whose faces is bounded by three or four edges.  Therefore $G$ is a $4$-framed graph with frame $G_0$.
\end{proof}

\subsection{Plane Map Graphs}

In addition to $1$-planar graphs, framed graphs can be used to obtain the following improved product structure theorem for (plane) $d$-map graphs.

\begin{thm}
\label{dMapProduct}
Every $d$-map graph is contained in $H \boxtimes P \boxtimes K_{ d + 3\floor{d/2} -3 }$ for some planar graph $H$ with $\tw(H) \leq 3$ and for some path $P$.
\end{thm}

\cref{dMapProduct} is an immediate consequence of \cref{d_framed_product_stucture} and the next lemma. Similar results to \cref{MewMapGraphLemma} appear in earlier works \cite{CGP06,BDGGMR,Brandenburg19,Brandenburg20} but we state the precise lemma and provide a proof here for the sake of completeness.

\begin{lem}
\label{MewMapGraphLemma}
For every integer $d\geq 3$, every $d$-map graph is a subgraph of a $d$-framed multigraph.
\end{lem}

\begin{proof}
Let $G_0$ be a graph embedded in the plane, with each face labelled a nation or a lake, and where each vertex of $G_0$ is incident with at most $d$ nations. Let $G$ be the corresponding map graph.

If $G_0$ has a face $F$ of length 2, then add a new vertex inside $F$ adjacent to both vertices on the boundary of $F$, which creates two new triangular faces $F_1$ and $F_2$. If $F$ is a lake, then make $F_1$ and $F_2$ lakes. If $F$ is a nation, then make $F_1$ a nation and make $F_2$ a lake. The resulting map graph is still $G$. We may therefore assume that $G_0$ has no face of length $2$.  We may also assume that $G_0$ is edge-maximal in the following sense: Adding any edge to $G_0$ will introduce a crossing or make it impossible to label the faces of $G_0$ as nations or lakes to obtain $G$ as the resulting map graph. This edge-maximal graph is well-defined since the assumption of having no face of length $2$ implies that $|E(G_0|\leq 3|V(G_0)|-6$.

We now argue that $G_0$ is biconnected.  Suppose that some face $F$ of $G_0$ has a disconnected boundary (implying that $G_0$ is disconnected). Let $v$ and $w$ be vertices in distinct components of the boundary of $F$. Add an edge $vw$ to $G_0$ inside of $F$. The corresponding map graph is unchanged, which contradicts the edge-maximality of $G_0$. Thus no face of $G_0$ has a disconnected boundary and therefore $G_0$ is connected.

Suppose that some face $F$ of $G_0$ has a repeated vertex $v$ in the boundary walk of $F$ and let $u,v,w$ be consecutive vertices in this boundary walk.  Since $v$ appears more than once in the boundary walk of $F$, the degree of $v$ is greater than $1$.  Since the degree of $v$ is greater than $1$ and the length of $F$ is greater than $2$, $u\neq w$. Since $v$ is repeated in the boundary walk of $F$, $uw$ is not an edge in the boundary of $f$. Add the edge $uw$ inside of $F$ and consider the resulting face $uvw$ to be a lake. Since $v$ appears elsewhere in the boundary of $F$, the corresponding map graph is unchanged, which contradicts the edge-maximality of $G_0$. Thus no facial walk of $G_0$ has a repeated vertex. Since each facial walk is connected, each face of $G_0$ is bounded by a cycle, and $G_0$ is biconnected.

Let $G_0^*$ be the dual graph of $G_0$. So the vertices of $G_0^*$ correspond to faces of $G_0$, and each vertex of $G_0^*$ is a nation vertex or a lake vertex. Since $G_0$ is biconnected, so is $G^*_0$.

Let $x$ be a vertex of $G_0$, let $F_x$ be the corresponding face of $G_0^*$, and let $v_1,\ldots,v_s$ be the facial cycle of $F_x$.  Let $C:=w_1,\ldots,w_r$ be the subsequence of $v_1,\ldots,v_s$ consisting of only the nation vertices.  Since $x$ is incident to at most $d$ nations, $r\le d$. Call $C_x$ the \defin{nation cycle} of $F_x$. Note that if $r=1$ then the ``nation cycle'' has no edges, and if $r=2$ then the ``nation cycle'' has one edge.

Let $G_1$ be a plane supergraph of $G_0^*$ obtained by adding the nation cycle of each face $F$ of $G_0$, and then triangulating each face with more than $d$ vertices on its boundary. $G_1$ is biconnected since $G^*_0$ is a biconnected spanning subgraph of $G_1$. Each nation cycle of length at least $3$ is now a facial cycle of $G_1$, and no face of $G_1$ has more than $d$ vertices on its boundary. Let $\widehat{G}_1$ be the edge-maximal $d$-framed graph whose frame is $G_1$.

By definition, $V(G) \subseteq V(\widehat{G}_1)$. To finish the proof, it suffices to show that $E(G)\subseteq E(\widehat{G}_1)$.  Indeed, if $vw\in E(G)$ then the nation faces corresponding to $v$ and $w$ have a common vertex $x$ on their boundary. The vertex $x$ corresponds to a face $F_x$ in $G_0^*$ and the facial cycle of $F_x$ contains $v$ and $w$.  Therefore, the nation cycle $C_x$ of $F_x$ contains $v$ and $w$. If $C_x$ has length $2$ then $vw\in E(G_1)\subseteq E(\widehat{G}_1)$.  If $C_x$ has length at least $3$ then $C_x$ bounds a face in $G_1$, so $vw\in E(\widehat{G}_1)$.
\end{proof}

%%%%%%%%%%%%%%%%%%%%%%%%
\section{Applications}
\label{Applications}

As discussed in the introduction, product structure has been used to resolve a number of problems on planar graphs.  In most cases, these results hold for any graph class with product structure. In this section, we survey some of the consequences of this for the graph classes considered in the previous two sections.

%%%%%%%%%%%%%%%%%%%
\subsection{Queue Layouts}

For an integer $k\geq 0$, a \defin{$k$-queue layout} of a graph $G$ consists of a linear ordering $\preceq$ of $V(G)$ and a partition $\{E_1,E_2,\dots,E_k\}$ of $E(G)$, such that for $i\in\{1,2,\dots,k\}$, no two edges in $E_i$ are nested with respect to $\preceq$. That is, it is not the case that $v\prec x \prec y \prec w$ for edges $vw,xy\in E_i$. The \defin{queue-number} of a graph $G$, denoted by \defin{$\qn(G)$}, is the minimum integer $k$ such that $G$ has a $k$-queue layout. Queue-number was introduced by \citet{HLR92}, who famously conjectured that planar graphs have bounded queue-number. \citet{dujmovic.joret.ea:planar} proved this conjecture using \cref{PlanarProduct} and the following lemma. Indeed, resolving this question was the motivation for the development of \cref{PlanarProduct}.

\begin{lem}[\citep{dujmovic.joret.ea:planar}]
\label{qn}
If $G\subseteq H \boxtimes P \boxtimes K_\ell$ then
$\qn(G) \leq  3 \ell \, \qn(H) + \floor{\tfrac{3}{2}\ell}
\leq 3 \ell \, 2^{\tw(H)}  - \ceil{\tfrac{3}{2}\ell}$.
\end{lem}

Since all the product structure theorems presented thus far upper bound the treewidth of $H$, \cref{qn} immediately implies bounds on the queue number of all graphs in these classes.  An improvement can be made for $1$-planar and $d$-map graphs since then $H$ is planar with treewidth at most $3$. \citet{ABGKP20} proved that every planar graph with treewidth at most $3$ has queue-number at most $5$. Thus the graph $H$ in \cref{1_planar_product,dMapProduct} has queue-number at most $5$.

The following corollary summarizes all of these results.
\begin{cor}\label{q_cor}
  The following bounds hold for the queue number of any graph from each of the following classes:
  \begin{compactitem}
    \item $k$-planar: $2^{O(k^3)}$
    \item $(g,k)$-planar: $\max\{1,g\}\cdot 2^{O(k^3)}$
    \item $(g,\delta)$-string: $\max\{1,g\}\cdot 2^{O(\delta^3)}$
    \item $1$-planar: $115$
    \item $d$-map: $\floor{ \frac{33}{2} (d+3\floor{\frac{d}{2}}-3) }$
    \item $(g,d)$-map: $\max\{1,g\}\cdot 2^{O(d)}$
    \item $k$-nearest-neighbour: $2^{O(k^6)}$
  \end{compactitem}
\end{cor}

Note that \citet{dujmovic.joret.ea:planar} previously proved the bound of
$O(g^{k+2})$ for $(g,k)$-planar graphs using \cref{GenusProduct} and an ad-hoc method. Our result provides a better bound when $g>2^{k^2}$.

%%%%%%%%%%%%%%%%%%%
\subsection{Non-Repetitive Colouring}

The next two applications are in the field of graph colouring. For our purposes, a \defin{$c$-colouring} of a graph $G$ is any function $\phi\colon V(G)\to C$, where $C$ is a set of size at most $c$.
A $c$-colouring $\phi$ of $G$ is \defin{non-repetitive} if, for every path $v_1,\ldots,v_{2h}$ in $G$, there exists $i\in\{1,\ldots,h\}$ such that $\phi(v_i)\neq\phi(v_{i+h})$.  The \defin{non-repetitive chromatic number} $\pi(G)$ of $G$ is the minimum integer $c$ such that $G$ has a non-repetitive $c$-colouring. This concept, introduced by \citet{AGHR-RSA02}, has since been widely studied; see \citep{Wood21} for an extensive survey. Up until recently the main open problem in the field was whether planar graphs have bounded non-repetitive chromatic number, first asked by \citet{AGHR-RSA02}. \citet{dujmovic.esperet.ea:planar} solved this question using \cref{PlanarProduct} and the following lemma.

\begin{lem}[\citep{dujmovic.esperet.ea:planar}]
\label{non-repetitive}
If $G\subseteq H\boxtimes P \boxtimes K_\ell$ then $\pi(G)\le \ell\, 4^{\tw(H)+1}$.
\end{lem}

Combining \cref{non-repetitive} with our results on product structure, we obtain the following corollary:
\begin{cor}\label{non-repetitive_cor}
  The following bounds hold for the non-repetitive chromatic number of any graph from each of the following classes:
  \begin{compactitem}
    \item $k$-planar: $(18k^2+48k+30)\cdot 4^{\binom{k+4}{3}}$
    \item $(g,k)$-planar: $\max\{2g,3\}\cdot(6k^2+16k+10) 4^{\binom{k+4}{3}}$
    % \item $(g,\delta)$-string: $g4^{O(\delta^3)}$
    \item $1$-planar: $1792$
    \item $d$-map: $256(d+3\floor{\frac{d}{2}}-3)$
    % \item $(g,d)$-map: $g4^{O(d)}$
    % \item $k$-nearest-neighbour: $4^{O(k^6)}$
    \item $(g,d)$-map graph: $(7d^2-21d)\cdot 4^{10}$
  \end{compactitem}
\end{cor}

Prior to the current work, the strongest upper bound on the non-repetitive chromatic number of $n$-vertex  $k$-planar graphs was $O(k\log n)$ \cite{dujmovic.morin.ea:layered}.  Our results also give bounds on the non-repetitive chromatic number of $(g,\delta)$-string graphs and $k$-nearest-neighbour graphs, but these bounds are weaker than existing results based on maximum degree.  The bounds obtained from product structure are exponential in $\delta$ and $k^2$, respectively.  However, each of these graph classes has degree bounded by $O(\delta)$ and $O(k^2)$, respectively, and therefore have non-repetitive chromatic number $O(\delta^2)$ and $O(k^{4})$, respectively \cite{DJKW16}.

%%%%%%%%%%%%%%%%%%%
\subsection{Centered Colourings}
\label{centered-colourings}

A $c$-colouring $\phi$ of $G$ is \defin{$p$-centered} if, for every connected subgraph $X\subseteq G$, $|\{\phi(v):v\in V(X)\}| > p$ or there exists some $v\in V(X)$ such that $\phi(v)\neq \phi(w)$ for every $w\in V(X)\setminus\{v\}$.  In words, either $X$ receives more than $p$ colours or some vertex in $X$ receives a unique colour.  Let $\chi_p(G)$ be the minimum integer $c$ such that $G$ has a $p$-centered $c$-colouring. Centered colourings are important since they characterise classes of bounded expansion, which is a key concept in the sparsity theory of \citet{Sparsity}.

% \citet{PS21} and
\citet{DFMS21} use \cref{PlanarProduct} and \cref{PlanarProduct}(b), respectively, to show that $\chi_p(G)\in O(p^2\log p)$ when $G$ is planar or of bounded Euler genus.  Upper bounds on $\chi_p$ for graphs of given treewidth~\citep{PS21} and for graph products~\citep{DFMS21} imply the next lemma.

\begin{lem}[\citep{DFMS21,PS21}]
\label{p-centered}
For every graph $H$ of treewidth at most $t$ and for every path $P$, if $G\subseteq H\boxtimes P \boxtimes K_\ell$ then
\[\chi_p(G)\le \ell (p+1)\, \chi_p(H) \leq \ell (p+1) \tbinom{p+t}{t}.\]
\end{lem}

For planar graphs of treewidth at most $3$ a $O(p^2\log p)$ upper bound is known:

\begin{lem}\cite{DFMS21}\label{simple_3tree_centered_colouring}
  For every planar graph $H$ of treewidth at most $3$,
  \[
    \chi_p(H)\le (p+1)(p\ceil{\log(p+1)}+2p+1) \enspace .
  \]
\end{lem}

Combining \cref{p-centered,simple_3tree_centered_colouring} with our results on product structure, we obtain the following corollary:

\begin{cor}\label{p_centered_cor}
  The following bounds hold for the $p$-centered chromatic number of any graph from each of the following classes:
  \begin{compactitem}
    \item $k$-planar: $(18k^2+48k+30)\cdot 4^{\binom{k+4}{3}}$
    \item $(g,k)$-planar: $\max\{2g,3\}\cdot(6k^2+16k+10)(p+1)\binom{p+\binom{k+4}{3}-1}{\binom{k+4}{3}-1}$
    % \item $(g,\delta)$-string: $g4^{O(\delta^3)}$
    \item $1$-planar: $7(p+1)^2(p\ceil{\log(p+1)}+2p+1)$
    \item $d$-map: $(d+3\floor{d/2}-3)\cdot (p+1)^2(p\ceil{\log(p+1)}+2p+1)$
    % \item $(g,d)$-map: $g4^{O(d)}$
    % \item $k$-nearest-neighbour: $4^{O(k^6)}$
    \item $(g,d)$-map: $\max\{2g,3\}\cdot 7d(d-3)(p+1)\cdot \binom{p+9}{9}$
  \end{compactitem}
\end{cor}

Prior to the current work, the strongest known upper bounds on the $p$-centered chromatic number of $(g,k)$-planar graphs $G$ were doubly-exponential in $p$, as we now explain. \citet{dujmovic.eppstein.ea:structure} proved that $G$ has layered treewidth $(4g+6)(k+1)$.
Van den Heuvel and Wood~\citep{vdHW18} showed this implies that $G$  has $r$-strong colouring number at most $(4g + 6)(k + 1)(2r + 1)$. By a result of \citet{zhu:colouring}, $G$ has $r$-weak colouring number at most $( (4g + 6)(k + 1)(2r + 1) )^r$, which by another result of  \citet{zhu:colouring} implies that $G$ has  $p$-centered chromatic number at most $( (4g+6)(k+1)(2^{p-1} + 1) )^{2^{p-2}}$. The above results are substantial improvements, providing bounds on $\chi_p(G)$ that are polynomial in $p$ for fixed $g$ and $k$.

As with non-repetitive chromatic number, our results also imply bounds on the $p$-centered chromatic numer of $(g,\delta)$-string graphs, powers of bounded-degree graphs, and $k$-nearest-neighbour graphs, but these bounds are weaker than the bounds implied by the fact that these graphs have bounded maximum degree \cite{DFMS21}.

\subsection{Universal Graphs}

As mentioned in \cref{Introduction}, \citet{DEJGMM21} and \citet{EJM} prove results about adjacency labelling schemes for planar graphs. Stated in terms of universal graphs, their main theorem is interpreted as follows:

\begin{thm}[\citep{DEJGMM21,EJM}]
  \label{Universal}
  For every fixed integer $t$ and every integer $n>0$ there exists a
  graph $U_n$ with $n^{1+o(1)}$ vertices and edges such that for every graph $H$ of treewidth at most $t$ and path $P$, every $n$-vertex subgraph of $H\boxtimes P$ is isomorphic to an induced subgraph of $U_n$.
\end{thm}

Combining \cref{Universal} with our results on product structure yields the following:

% this with \cref{gkPlanarProduct,k-nn,PowerMinor,StringPartition,MapPartition} yields the following:

\begin{cor}\label{universal_cor}
  \label{UniversalUniversal}
  For every fixed graph $X$ and all fixed integers $d,\delta,\Delta,g,k>0$, and every integer $n>0$, there exists a graph $U_n$ with $n^{1+o(1)}$ vertices and edges such that $U_n$ contains the following graphs as induced subgraphs:
  \begin{compactitem}
    \item every $n$-vertex $(g,k)$-planar graph;
    \item every $n$-vertex $(g,d)$-map graph;
    \item every $n$-vertex $(g,\delta)$-string graph;
    \item every $n$-vertex graph $G^k$ where $G$ is $X$-minor-free and has maximum degree at most $\Delta$;
    \item every $k$-nearest neighbour graph of $n$ points in $\R^2$.
  \end{compactitem}
\end{cor}

\subsection*{Acknowledgements} Thanks to the referees for many helpful comments.

% %%%  Squashing the bibliography
   \let\oldthebibliography=\thebibliography
   \let\endoldthebibliography=\endthebibliography
   \renewenvironment{thebibliography}[1]{%
     \begin{oldthebibliography}{#1}%
       \setlength{\parskip}{0ex}%
       \setlength{\itemsep}{0ex}%
   }{\end{oldthebibliography}}

\bibliographystyle{DavidNatbibStyle}
\bibliography{k-planar}
\end{document}